\documentclass[a4paper,12pt]{article} 
\usepackage{amsfonts,amssymb,amsthm,amsmath}
 \usepackage[dvips]{graphics}
\usepackage{subfigure}
 \newcounter{idesc} 
\usepackage{color}

\DeclareSymbolFont{AMSb}{U}{msb}{m}{n}

\begin{document}

\def\A{\mathcal {A}}
\def\B{\mathcal {B}}
\def\C{\mathcal {C}}
\def\P{\mathbb P}
\def\N{\mathbb N}
\def\HH{\mathbb H}
\def\VV{\mathbb V}
\def\H{\mathcal {H}}
\def\K{\mathcal {K}}
\def\M{\mathcal {M}}
\def\R{\mathbb {R}}
\def\TT{\mathbb {T}}
\def\RR{\mathcal {R}}
\def\U{\mathcal {U}}
\def\DE{D(\mathcal E)}
\def\F{\mathcal {F}}
\def\EE{\mathcal {E}}
\def\E{\mathbb {E}}
\def\1{\,{\makebox[0pt][c]{\normalfont    1}
\makebox[2.5pt][c]{\raisebox{3.5pt}{\tiny {$\|$}}}
\makebox[-2.5pt][c]{\raisebox{1.7pt}{\tiny {$\|$}}}
\makebox[2.5pt][c]{} }}
\def\one{\1 }
\def\eI{[0,1]}
\def\eps{\varepsilon}
\def\prt{\partial}
 \def\one{\1 }
 \newcommand{\norm}[1]{\left\| #1 \right\|}

\newtheorem{thm}{Theorem}
\newtheorem{cor}[]{Corollary}
\newtheorem{defn}[]{Definition}
\newtheorem{bem}[]{Remark}
\newtheorem{bemn}[bem]{Remarks}
\newtheorem{bsp}[]{Example}
\newtheorem{bspe}{Examples}

\newcommand\iitem{\stepcounter{idesc}\item[(\roman{idesc})]}
\newenvironment{idesc}{\setcounter{idesc}{0}
\begin{description}} {\end{description}}

\def\dnvol{{d \ol{\normalfont v}}}

\def\Tod{\mathop{\Rightarrow}\limits^d}
\def\PP{{\mathcal P}}
\def\Z{{\mathbb Z}}
\def\CMS{CM[0,1]^2}
\def\eI{{[0,1]}}
\def\L2{{L^2}}
\def\L11{\mathcal P _{ac}}
\def\llangle{\langle \langle}
\def\rrangle{\rangle \rangle}
\def\dd{\mathrm{d}}

 \def\Pbeta{\mathbb P^\Beta}
 \def\G{\mathcal G}
 \def\Qbeta{{\mathbb Q^\beta}}
  \def\bbox{{\hfill $\Box$}}

  \def\C{{\mathcal C}}
 \def\D{{\mathcal D}}

\newcommand{\comp}{\mathrm{comp}\,}
\newcommand{\leb}{{\mbox{Leb}}}
\newcommand{\Cyl}{{\mathfrak{C}}}
\newcommand{\Syl}{{\mathfrak{S}}}
\newcommand{\Zyl}{{\mathfrak{Z}}}
\def\smint{{\mbox{$\int$}}}

  \newcommand{\Ent}{\mbox{\rm Ent}}
 \newcommand{\Gaps}{\mbox{\rm gaps}}

\newtheorem{theorem}{Theorem}[section]
\newtheorem{lemma}[theorem]{Lemma}
\newtheorem{corollary}[theorem]{Corollary}
\newtheorem{proposition}[theorem]{Proposition}

\theoremstyle{definition}
\newtheorem{remark}[theorem]{Remark}
\newtheorem{example}[theorem]{Example}
\newtheorem{definition}[theorem]{Definition}

\newcommand{\Q}{{\mathbb{Q}}}
\newcommand{\Pp}{{\mathbb{P}}}

\renewcommand{\labelenumi}{(\roman{enumi})}

\author{Alexei Kulik
\footnote{Institute of Mathematics, NAS of Ukraine, 3, Tereshchenkivska str., 01601  Kyiv, Ukraine
\sf{kulik.alex.m@gmail.com}},
Michael Scheutzow
\footnote{Institut f\"ur Mathematik, MA 7-5,
Technische Universit\"at Berlin,
10623 Berlin,
Germany
\sf{ms@math.tu-berlin.de}}}

\title{Generalized couplings and convergence of transition probabilities\\
}

\maketitle  ~\\

\begin{abstract}
We provide sufficient conditions for the uniqueness of an invariant measure of a Markov process as well as for the weak convergence of transition
probabilities to the invariant measure. Our conditions are formulated in terms of generalized couplings. We apply our results to several SPDEs for
which unique ergodicity has been proven in a recent paper by Glatt-Holtz, Mattingly, and Richards and show that under essentially the same assumptions
the weak convergence of transition probabilities actually holds true.
\end{abstract}


\section{Introduction}

In this article, we provide sufficient conditions in terms of (generalized) couplings for  the uniqueness of an invariant measure and weak
convergence to the invariant measure for a Markov chain taking values in a Polish space $E$. Such criteria have already been established for the
uniqueness of an invariant measure in \cite{BM05} and \cite{HMS11} but -- to the best of our knowledge -- not for the  weak convergence
(or asymptotic stability) of transition probabilities. In \cite{KPS10} and \cite{BKS}, uniqueness and asymptotic stability were shown
for so-called e-processes (which we explain below)
and a similar approach was used in \cite{HMS11} to prove asymptotic stability. Our aim is to present a unified approach to both uniqueness and asymptotic
stability in terms of generalized (asymptotic) couplings. Here, a probability measure $\xi$ on a product space is called a {\em generalized} coupling of
$\mu$ and $\nu$ if the marginals of $\xi$ are not necessarily equal to $\mu$ and $\nu$ but only absolutely continuous. We point out that we
do not assume the e-property to hold (which indeed does not hold in all cases of interest -- see e.g. Example \ref{ex54} --  and even if it does it is often
cumbersome to verify). Our uniqueness statement, Theorem  \ref{unique}, is a slight generalization of \cite[Theorem 1.1]{HMS11}. We will show in
Example \ref{comparison} that our conditions are indeed strictly weaker than those in \cite{HMS11}. At the same time the proof is quite short and elementary.

We then proceed to formulate sufficient conditions for the convergence of transition probabilities assuming existence of an invariant measure $\mu$.
Our main results are Theorems \ref{conv1} and \ref{conv2}, the former one providing a sufficient condition for weak convergence of transition probabilities for
$\mu-$almost all initial conditions and the latter one for weak convergence of the transiton probabilities starting from a given point $x \in E$. In both theorems,
the conditions are formulated in terms of generalized couplings. In Theorem \ref{conv1} convergence is in fact in probability, i.e. the measure $\mu$ of the
set of initial  conditions for which the distance of the transition probability to the invariant measure $\mu$ after $n$ steps is larger than $\eps$ converges to 0
for every $\eps>0$. It seems to be an open question if convergence even holds true in an almost sure sense. We point out that our proof of Theorem \ref{conv1}
requires the chain to be {\em Feller}, while Theorems \ref{unique} and \ref{conv2} do not rely on this property. However, the Feller property is required to hold true
with respect to some metric $d$ on $E$ which in general may differ from the original one, hence this assumption is quite flexible and non-restrictive.

The main results are proved in Sections \ref{s3} and \ref{s4}. In Section \ref{s5} several examples are given which illustrate the conditions imposed in the main
results and clarify the relations of these results with some other available in the field.

To illustrate the usefulness of our results, in Section \ref{sSDDE} we give two groups of their applications. First, in Section \ref{s31}, we consider  the same
example as in \cite{HMS11}, namely a stochastic delay  equation which has the space of continuous functions on $[-1,0]$ as its natural state space.
The solution is a Feller process and it is not hard to find a generalized coupling which satisfies the conditions of Theorem \ref{unique}. This coupling is actually a simplified version of the one used in the proof of \cite[Theorem 3.1]{HMS11}; namely, our  Theorem \ref{unique}, unlike \cite[Theorem 1.1]{HMS11}, does not require the equivalence of the marginal distributions, thus the ``localization in time'' part of the construction of the generalized coupling can be omitted. Remarkably, such a simplified construction appears well applicable in Theorems \ref{conv1} and \ref{conv2} as well, so that without any additional work we get asymptotic stability
for free (unlike in \cite{HMS11}).

This method to improve a result from unique ergodicity to asymptotic stability looks quite generic, and in our second group of applications essentially the same method
leads to  several new statements. We reconsider the results from the recent paper  \cite{GMR15}, where   the generalized coupling technique
(or {\em asymptotic coupling}, in their terminology) is used to prove uniqueness of an ergodic measure for several types of non-linear SPDEs. Each of the five
SPDE models considered therein is analytically quite involved, and \cite{GMR15} perfectly illustrates the flexibility of the generalized coupling approach, which
appears to be well applicable in complicated models. In Section \ref{s32} we show that just  minor modifications in the construction of the generalized couplings
from \cite{GMR15} make them applicable in our Theorems \ref{conv1} and \ref{conv2}, as well, providing asymptotic stability (almost) for free and thus illustrating
the power of our approach.

\section{Main results}\label{mainresults}
\subsection{Basic definitions and notation}\label{basic}
Let $(E,\rho)$ be a Polish (i.e.~separable, complete metric) space with Borel $\sigma$-algebra $\EE$, and let $X=\{X_n, n\in {\mathbb{Z}}_+\}$ be a Markov chain with
state space $(E, \mathcal{E})$, where
${\mathbb{Z}}_+:=\{0,1,\cdots\}$. Transition probabilities  and $n$-step transition probabilities for $X$ are denoted respectively by $P(x,\dd y)$ and $P_n(x,\dd y)$.
Let $E^\infty:=E^{{\mathbb{Z}}_+}$. The law of the sequence $\{X_n\}$ in $(E^\infty, \mathcal{E}^{\otimes \infty})$ with initial distribution $\mathrm{Law}\, (X_0)=\mu$
is denoted by $\mathbb{P}_\mu$, the respective expectation is denoted by $\mathbb{E}_\mu$; in case $\mu=\delta_x$ we write
simply $\mathbb{P}_x, \mathbb{E}_x$.

Recall that an invariant probability measure for $X$ is a probability measure $\mu$ on $(E, \mathcal{E})$ such that
\begin{equation}\label{invar}
\mu(\dd y)=\int_EP(x,\dd y)\mu(\dd x).
\end{equation}
Equivalently, a probability measure $\mu$  is invariant if the sequence $\{X_n, n\in \mathbb{Z}_+\}$ is strictly stationary under $\mathbb{P}_\mu$. An invariant
probability
measure $\mu$ for $X$ is {\em ergodic}, if the left shift on the space $(E^\infty, \mathcal{E}^{\otimes \infty})$ is ergodic with respect to  $\mathbb{P}_\mu$.
Recall that a strictly stationary sequence $\zeta_n, n\in \mathbb{Z}_+$ is called \emph{mixing} if for any bounded measurable functions $f,g:E\to \R$
\begin{equation}\label{mixing}
\E f(\zeta_0)g(\zeta_{n})\to \E f(\zeta_0)\E g(\zeta_0), \quad n\to \infty.
\end{equation}

For a measurable space $(S, \mathcal{S})$, we denote the set of all probability measures on  $(S, \mathcal{S})$ by $\mathcal{P}(S)$.
For given $\mu, \nu\in \mathcal{P}(S)$, define
$$
C(\mu,\nu)=\Big\{\xi\in \mathcal{P}(S\times S):\pi_1(\xi)=\mu,\, \pi_2(\xi)=\nu\Big\},
$$
where $\pi_i(\xi)$ denotes the $i$-th marginal distribution  of $\xi, i=1,2$. Any  $\xi\in C(\mu,\nu)$ is called a \emph{coupling} for
$\mu,\nu.$ We also introduce the following two extensions of the notion of a coupling. Recall that $\mu \ll \nu$ means that $\mu$ is absolutely continuous
with respect to
$\nu$ and $\mu \sim \nu$ means that $\mu$ and $\nu$ are equivalent, i.e. mutually absolutely continuous. Define
$$
\widetilde C(\mu,\nu)=\Big\{\xi\in \mathcal{P}(S\times S):\pi_1(\xi)\sim \mu, \pi_2(\xi)\sim\nu\Big\},
$$
$$
\widehat  C(\mu,\nu)=\Big\{\xi\in \mathcal{P}(S\times S):\pi_1(\xi)\ll\mu, \pi_2(\xi)\ll\nu\Big\},
$$
and call any probability measure  from one of the classes $\widetilde C(\mu,\nu), \widehat  C(\mu,\nu)$  a \emph{generalized coupling} for
$\mu,\nu.$
 In what follows, we use this notation mainly in the following two frameworks: (a) $S=E$, (b) $S=E^\infty$; that is, the probability  measures  are considered
 either on
 the initial state space or on the trajectories space. To distinguish the notation, we denote probability measures by $\mu,\nu, \dots$ and
 $\P,\Q, \dots$
 respectively in the first and the second cases.

 For given $p>1$ and $R \ge 1$, denote by $\widehat  C^R_p(\P,\Q)$ the set of generalized couplings  $\xi\in \widehat  C(\P,\Q)$ such that
$$
\left(\int_{E^\infty}\left({\dd \pi_1(\xi)\over \dd\P}\right)^p\,\dd\P\right)^{1/p}\leq R, \quad
\left(\int_{E^\infty}\left({\dd\pi_2(\xi)\over \dd\Q}\right)^p\,\dd\Q\right)^{1/p}\leq R.
$$

 Let $(S, \rho)$ be a metric space and $h:S\times S\to [0,1]$ be a \emph{distance-like function}; that is, $h$ is symmetric, lower semicontinuous, and
$h(x,y)=0\Leftrightarrow x=y$. The associated  \emph{minimal} (or \emph{coupling}) distance on  $\mathcal{P}(S)$ (denoted by the same letter $h$) is defined by
$$
h(\mu,\nu)=\inf_{\eta\in C(\mu,\nu)}\int_{S\times S} h(x,y)\, \eta(\dd x,\dd y), \quad \mu,\nu\in \mathcal{P}(S).
$$
When $\rho\leq 1$ and $h(x,y)=\rho(x,y),$ the above definition coincides with the definition of the Kantorovich-Rubinshtein metric (also commonly
called 1-Wasserstein metric) on $\mathcal{P}(S)$, and it is well known that $\mathcal{P}(S)$ with this metric is a Polish space; cf. \cite{Dudley}, Chapter 11.

Without loss of generality we assume furthermore that the metric $\rho$ on $E$ satisfies $\rho\leq 1$ (otherwise we introduce an equivalent metric $\rho\wedge 1$).
We  define the metric $\rho^{(\infty)}$ on $E^\infty$ by
$$
\rho^{(\infty)}(x,y)=\sum_{n=0}^\infty{1\over 2^{n+1}}d(x_n,y_n), \quad  x=(x_n)_{n\geq 0}, \, y=(y_n)_{n\geq 0}\in E^\infty,
$$
and   the metric $\rho^{(\infty,\infty)}$ on $E^\infty\times E^\infty$ by
$$
\rho^{(\infty,\infty)}\Big((x,x'), (y, y')\Big)=\rho^{(\infty)}(x,y)+\rho^{(\infty)}(x', y').
$$
We consider $\mathcal{P}(E)$, $\mathcal{P}(E^\infty)$, and $\mathcal{P}(E^\infty\times E^\infty)$ as Polish spaces w.r.t.~the corresponding Kantorovich-Rubinshtein
metrics  $\rho,\rho^{(\infty)}$ and $\rho^{(\infty,\infty)}$. The metric $\rho$ on $\mathcal{P}(E)$ induces weak convergence which we will denote by $\Rightarrow$.

Recall the following facts about the structure of the set of invariant probability measures;  e.g. \cite{DZ96}, Section 3.2 or \cite[Theorem 5.7]{H06}:
\begin{itemize}
  \item The set $\mathcal{I}_X$ of the invariant probability measures for $X$ is a convex compact set in $\mathcal{P}(E)$.
  \item Each two different ergodic invariant probability measures are mutually singular.
  \item Every extreme point of the set $\mathcal{I}_X$ is an ergodic invariant probability  measure,
  and each invariant probability measure $\mu$ has a representation of the form
  \begin{equation}\label{decomp}
\mu=\int_{\mathcal{P}(E)} \nu \, \kappa (\dd\nu),
  \end{equation}
where $\kappa$ is a probability measure
on the space $\mathcal{P}(E)$ which is concentrated on the extreme points  of the set $\mathcal{I}_X$.
\end{itemize}

Together with the initial metric $\rho$  on $E$, we will consider another metric $d$, and assume that it is bounded and continuous  with respect to the metric $\rho$.
The metric $d$ is not assumed to be complete. All measurability and continuity statements refer to $\rho$ rather than $d$ unless we explicitly say something different.
Considering two metrics, one to deal with measurability issues and one to prove convergence, is motivated by applications to SPDE models, see Section \ref{s32} below.
In many cases of interest however one can avoid such  complications and choose $d$ and $\rho$ the same.

When $d\not=\rho$, we denote by $\overline{E}^d$ the completion of $E$ with respect to $d$, and regard  $E$ as a subset in $\overline{E}^d$.
Note that $(\overline{E}^d,d)$ is a Polish space (where we denote the extended metric again by $d$). We also assume that for any $y\in E$ there exist a sequence of $d$-continuous functions $\rho_n^y:\overline{E}^d\to [0, \infty)$ such that for $x\in  \overline{E}^d$
\begin{equation}\label{approx}
\rho_n^y(x)\to \left\{
                 \begin{array}{ll}
                   \rho(x,y), & x\in E \\
                   \infty, & \hbox{otherwise}
                 \end{array}
               \right.,\quad n\to \infty;
\end{equation}
 cf. \cite{GMR15}, Appendix A.  This ensures that the image in $\overline{E}^d$ of any open ball in $E$ is a $d$-Borel subset. Because $(E,d)$ is separable, this implies that $E$ is a $d$-Borel subset in $\overline{E}^d$ and guarantees that the \emph{trace $\sigma$-algebra} 
 $$
 \{A\cap E, A\in \mathcal{B}(\overline{E}^d)\} 
 $$
 on $E$ coincides with $\EE$, hence
allowing us to identify $\mathcal{P}(E)$ with the set of measures from
$\mathcal{P}(\overline{E}^d)$ which provide a full measure for $E$.

We will use a separate notation $\Tod$ for the weak convergence in $\mathcal{P}(E)$ with respect to $d$. We will call the  Markov chain  {\em $d$-Feller}
if for each bounded and $d$-continuous function $f:E \to \mathbb{R}$, the map $x \mapsto \int f(y)P(x,\dd y)$ is $d$-continuous.

Finally, recall that $X$ is called an \emph{e-chain} with respect to the metric $d$, if its transition probability function is \emph{$d$-equicontinuous}:
for any $x\in E, \eps>0$ there exists $\delta>0$ such that
$$
d(P_n(x,\cdot), P_n(y,\cdot))\leq \eps, \quad n\geq 0, \quad d(x,y)<\delta.
$$
We note that our definition  is equivalent to that in \cite[Definition 2.1]{KPS10} by the Kantorovich-Rubinshtein duality theorem, see \cite{Dudley}, Chapter 11.

\subsection{Main theorems}

Our first main result is aimed at uniqueness of an invariant probability measure. It is a slight generalization of \cite[Theorem 1.1]{HMS11}.

\begin{theorem}\label{unique} Let $\mu_1$ and $\mu_2$ be ergodic invariant probability measures. Assume that for  some set
$M\in \EE\otimes \EE$ with
$\mu_1 \otimes \mu_2 (M)>0$  for each $(x,y)\in M$ there exists $\alpha_{x,y}>0$ such that
for each $\varepsilon>0$ there exists  $\xi_{x,y}^\eps \in \widehat  C(\P_x,\P_y)$  which satisfies
\begin{equation}\label{limsup}
\limsup_{n \to \infty} \frac 1n \sum_{i=0}^{n-1}\xi_{x,y}^\eps \big( d(X_i,Y_{i})\le \varepsilon\big) \ge \alpha_{x,y}.
\end{equation}\
Then $\mu_1=\mu_2$.
\end{theorem}

Theorem \ref{unique} combined with the representation \eqref{decomp} immediately implies the following statement.

\begin{corollary}\label{corounique}
Let   $M \in \EE$ be  such that $\mu(M)>0$ for every invariant probability measure $\mu$ and for each $x,y\in M$ there exists some
$\alpha_{x,y}>0$ such that
for each $\varepsilon>0$ there exists  $\xi_{x,y}^\eps  \in \widehat  C(\P_x,\P_y)$  which satisfies
\begin{equation}\label{limsupnew}
\limsup_{n \to \infty} \frac 1n \sum_{i=0}^{n-1}\xi_{x,y}^\eps \big( d(X_i,Y_{i})\le \varepsilon\big) \ge \alpha_{x,y}.
\end{equation}

Then there exists at most one invariant probability measure, and if there exists one this measure is ergodic.
\end{corollary}

The second main result provides the weak convergence of transition probabilities to an invariant probability measure $\mu$ in a somewhat  unusual  form of
the ``weak convergence in probability''.

\begin{theorem}\label{conv1}
I. Assume that  $X$ is $d$-Feller. Let $\mu$ be an  invariant probability measure, and assume  that for some $M \in \EE\otimes \EE$ with  $(\mu\otimes \mu)(M)=1$
the following condition holds true for each $(x,y)\in M$:
\begin{itemize}
\item[\rm{(i)}] $$\lim_{\varepsilon \downarrow 0+}\liminf_{n \to \infty} \sup_{\xi \in C(\P_x,\P_y)} \xi\big( d(X_n,Y_n) \le \varepsilon \big) >0. $$
\end{itemize}
Then
\begin{equation}\label{weak_in_prob}
\mu\left(x:d\Big(P_n(x,\cdot),\mu\Big)>\eps\right)\to 0, \quad n\to \infty, \quad \eps>0;
\end{equation}
that is, $P_n(\cdot,\cdot),\, n\geq 0$, considered as a sequence of $\mathcal{P}(E)$-valued random elements on $(E, \mathcal{E},\mu)$, $d$-converges in probability
to the random element identically equal to $\mu$.

II. The following condition is equivalent to (i):\
\begin{itemize}
\item[\rm{(ii)}] For some $p >1$, $R \ge 1$,
$$
\lim_{\varepsilon \downarrow 0+}\liminf_{n \to \infty} \sup_{\xi \in \widehat  C^R_p(\P_x,\P_y)} \xi\big( d(X_n,Y_n) \le \varepsilon \big) >0. $$
\end{itemize}
Further, the following condition implies (ii) (and (i)):
\begin{itemize}
\item[\rm{(iii)}] $$\sup_{\xi \in \widehat  C(\P_x,\P_y)}\lim_{\varepsilon \downarrow 0+}\liminf_{n \to \infty} \xi \big( d(X_n,Y_n)\le \varepsilon\big) >0.$$
\end{itemize}
\end{theorem}

Combining the two statements of the theorem, we directly get the following corollary, formulated in terms similar to those used in Theorem \ref{unique}.

\begin{corollary}\label{coroconverge} Let $X$ be $d$-Feller, $\mu$ be an ergodic invariant probability measure and  $M \in \EE \otimes \EE$ be such that
$\mu \otimes \mu(M)=1$,
and assume that for each $(x,y) \in M$ there exist $\xi_{x,y}\in \widehat C(\P_x,\P_y)$ and $\alpha_{x,y}>0$ such that
$$
\liminf_{n \to \infty} \xi_{x,y} \big( d(X_n,Y_n)\le \varepsilon\big) \geq\alpha_{x,y}
$$
for every $\eps>0$.
Then \eqref{weak_in_prob} holds true.
\end{corollary}

In the following theorem  a stronger (and more typical) type of convergence is obtained at the cost of making a stronger assumption: the e-chain property
(which is essentially the uniform Feller property) instead of the usual Feller one.
\begin{theorem}\label{conv11} Let $X$ be an e-chain w.r.t. $d$, and one of the assumptions (i) -- (iii) of Theorem \ref{conv1} hold true.

Then $P_n(x,\cdot)\Tod \mu$ for $\mu$-a.a. $x\in E$.
\end{theorem}

A proper benchmark for Theorem \ref{conv11} is  the \emph{modified Doob theorem}, given in \cite[Theorem 2]{KS15}. Let, for a while, $d(x,y)=1_{x\not=y}$
be the discrete metric (which is however is not included into our setting since it is not continuous). Then by the \emph{Coupling Lemma} (e.g. \cite[Lemma 1]{KS15}) the corresponding probability distance equals $1/2$ of the total
variation distance. In \cite[Theorem 2]{KS15} it is assumed that for $\mu\otimes\mu$-a.a. $(x,y)$ there exists $n=n_{x,y}$ such that
$P_n(x, \cdot)\not\perp P_n(y, \cdot)$, which by the Coupling Lemma is equivalent to existence of a coupling $\xi_{x,y}$ and positive $\alpha_{x,y}$ such that
$$
\xi_{x,y} \big( d(X_n,Y_n)=0\big)\geq \alpha_{x,y}.
$$
One can extend the coupling $\xi_{x,y}$ in such a way that $X_N=Y_N, N\geq n$, hence the above assumption actually coincides with the one from
Corollary \ref{coroconverge}. That is, Theorem \ref{conv11} is a direct analogue of the modified Doob theorem, which operates with weak convergence
of the  transition probabilities instead of total variation convergence.

Note that the discrete metric $d$ is \emph{non-expanding}: since the discrete metric $d$ takes values $0,1$ only,
$$
d(P_n(x,\cdot), P_n(y,\cdot))\leq d(x,y), \quad x,y\in E, \quad n\geq 1.
$$
This property  has the same meaning as the e-chain property, which in Theorem \ref{conv11} is imposed as an additional assumption because  general
 metric $d$  may fail to be non-expanding. The ergodicity under the e-chain (actually, the e-process) property was systematically
studied in \cite{BKS}, \cite{KPS10}, see also \cite[Theorem 3.7]{HMS11}, where  the e-chain property was used essentially without naming it explicitly.
We remark that the e-chain property, although being quite typical for ergodic processes, \emph{does not follow} from the fact that the transition
probabilities converge to the (unique) invariant probability distribution: see Example \ref{ex54} below, which in particular shows that
Proposition 6.4.2 in \cite{MeTwee} is incorrect.
In that concern, the clearly seen advantage of Theorem \ref{conv1} is that there we avoid  the quite non-elementary (and sometimes not easy to verify)
e-chain assumption. We remark that both of the proofs of Theorem \ref{conv1} and Theorem \ref{conv11} exploit the  typical  ``coupling''  idea.
Namely, we make one ``coupling attempt''
with the probability of success being close to the presumably maximal possible one, and then we show that if the latter probability is $<1$,
another ``coupling attempt'' will increase the overall probability of success significantly. In that strategy of the proof,
a kind of the  ``non-expansion'' property is crucial in order to preserve the positive result of the first ``coupling attempt''.
We note that, in the proof of Theorem \ref{conv1}, only the basic $d$-Feller property is used to provide such ``non-expansion''.

Theorem \ref{conv1}  provides the following important corollary.

\begin{corollary}\label{coromixing} Under the conditions of Theorem \ref{conv1}, the (stationary) chain $X$ is  mixing w.r.t. $\P_\mu$.
\end{corollary}

This corollary gives a good  prerequisite for our  third main result, which provides a sufficient condition for the transition probabilities $P_n(x,.)$ of a given $x \in E$ to converge to the invariant measure $\mu$.

\begin{theorem}\label{conv2}
Let $\mu$ be an  invariant probability measure and $X$ be mixing w.r.t. $\P_\mu$. Fix $x \in E$ and
assume that there exists a set $M \in \EE$ such that $\mu(M)>0$ and for every $y \in M$ there exists some $\xi_{x,y} \in \widehat C(\P_x,\P_y)$  such that
$\pi_1(\xi_{x,y})\sim \P_x$ and
\begin{equation}\label{prob1}
\lim_{n \to \infty} \xi_{x,y} \big( d(X_n,Y_n) \le \varepsilon\big) =1
\end{equation}
for every $\varepsilon>0$.
Then $P_n(x,\cdot)\Tod \mu$.
\end{theorem}

Combining Corollary \ref{coroconverge}, Corollary \ref{coromixing}, and Theorem \ref{conv2} we easily derive the following corollary, which provides
weak convergence of $P_n(x, \cdot)$ for any starting point $x$ in terms of generalized couplings.

\begin{corollary}\label{coroconv} Let $X$ be $d$-Feller, and assume that for any $(x,y) \in E\times E$ there exists some $\xi_{x,y} \in \widehat C(\P_x,\P_y)$  such that
$\pi_1(\xi_{x,y})\sim \P_x$ and \eqref{prob1} holds true for every $\varepsilon>0$.

Then there exists at most one invariant probability measure, and if such a measure $\mu$ exists then $P_n(x,\cdot)\Tod \mu$ for any $x\in E$.
\end{corollary}

\begin{remark}\label{rem29} The assumption of Corollary \ref{coroconv} is well designed to be easily applied in various particular settings; we illustrate this in Section
\ref{sSDDE} below considering Markov chains generated by stochastic functional delay equations (SFDEs) and stochastic partial differential equations (SPDEs).
On the other hand, this assumption is quite precise and can not  be essentially weakened. For instance, the required statement may fail
if one assumes only \eqref{prob1} for some $\xi_{x,y} \in \widehat C(\P_x,\P_y)$ without the additional condition $\pi_1(\xi_{x,y})\sim \P_x$,
see Example \ref{ex55} below.
\end{remark}

\begin{remark} The \emph{existence} of an invariant probability measure is a much easier topic, studied in great detail in the literature, and we do not address it here,
referring e.g. to \cite{DZ96}.
\end{remark}

\begin{remark}\label{rem_cont} All the main statements, formulated above in the discrete-time case, have  straightforward analogues in the
 continuous-time case. Namely, if it is assumed that the process $X_t, t\in [0,\infty)$ for any  $x$ has a c\`adl\`ag modification with $X_0=x$,
we can repeat the arguments literally, with the space $E^\infty$ changed to $\mathbb{D}([0,\infty), E)$ and $\P_x, x\in E$ being the respective laws
of $X$ in $\mathbb{D}([0,\infty), E)$. We remark that in some specific but important cases it may happen that the Markov process
$X_t, t\in [0,\infty)$ is stochastically continuous, but fails to have a   c\`adl\`ag modification. For such an example in the framework of L\'evy
driven SPDEs we refer to \cite{6Authors}. In this case the proofs of Theorem \ref{conv1}, Theorem \ref{conv11}  and Theorem \ref{conv2} also can be
adapted, and the statements of these theorems and respective Corollaries \ref{coroconverge}, \ref{coromixing}, and  \ref{coroconv}  hold true.
The technical difficulty which arise here is that now we do not have a ``good'' space of trajectories, hence the statements of Proposition
\ref{jan_cor} and Proposition \ref{prop} on the measurable choice can not be applied directly. These statements can be modified properly,
but in order not to overburden the exposition we will not go into further details.
\end{remark}

\section{Proofs of Theorem \ref{unique} and Theorem \ref{conv2}}\label{s3} In this section we provide  proofs of two of our main results, which are comparatively
simple and  are mainly based on the  Ergodic theorem.

\begin{proof}[Proof of Theorem \ref{unique}]
 If $\mu_1\not=\mu_2$ then $\mu_1\perp\mu_2$.
The fact that $E$ is Polish implies that for any probability measure $\mu$ and each set $A \in \EE$, we have $\mu(A)=\sup \mu(K)$, where the supremum is taken
over all  compact subsets of $E$ which are contained in $A$ (this is sometimes called {\em Ulam's theorem} or {\em inner regularity} of $\mu$; e.g. \cite{Billingsley},
Theorems 1.1 and 1.4). Therefore,
for every $m\geq 1$ there exist compact sets $K^m_{1,2}$ such that $K^m_1\cap K^m_2=\emptyset$ and $\mu_i(K_i^m)>1-1/m$. Since $d$ is continuous and $K^m_{1,2}$ are compact and disjoint, the $d$-distance between $K^m_1$ and $K^m_{2}$ is positive.
Then there exists a $d$-Lipschitz function $f^m:E \to [0,1]$ such that $f^m|_{K_1^m}\equiv 0$, $f^m|_{K_2^m}\equiv 1$.

Choose $(x,y)  \in M$, and let $\alpha_{x,y}$ be  as in the statement of the theorem. We can and will assume in addition that  $x,y$ are chosen such that
\begin{equation}\label{add}
\frac 1n \sum_{i=0}^{n-1}f^m (X_i)\to \int f^m \, \dd\mu_1\quad \hbox{a.s. w.r.t. }\mathbb{P}_x, \quad \frac 1n\sum_{i=0}^{n-1}f^m (Y_{i})\to \int f^m\,
\dd\mu_2\quad \hbox{a.s. w.r.t. }\mathbb{P}_y
\end{equation}
for every $m\geq 1$. Take  $m_0>2/\alpha_{x,y}$ and fix $\varepsilon >0$ such that
$$
\varepsilon <\frac{\alpha_{x,y}-2/m_0}{\alpha_{x,y} \mathrm{Lip}(f^{m_0})}.
$$
Let $\xi_{x,y}^\eps  \in \widehat  C(\P_x,\P_y)$ be as in the statement of the theorem, then
$$
\liminf_{n\to \infty}\E^{\xi_{x,y}^\eps}\left(\frac 1n \sum_{i=0}^{n-1}(f^{m_0} (Y_{i})-f^{m_0} (X_i))\right)\leq (1-\alpha_{x,y})+\varepsilon \mathrm{Lip}(f^{m_0})
\alpha_{x,y}<1-2/{m_0}.
$$

Because the distribution of $\{X_i\}$ (resp. $\{Y_i\}$) w.r.t. $\xi_{x,y}^\eps$ is absolutely continuous w.r.t. $\mathbb{P}_x$ (resp. $\mathbb{P}_y$) and $f^{m_0}$ is
bounded, it follows from (\ref{add}) that
$$
\E^{\xi_{x,y}^\eps}\left(\frac 1n\sum_{i=1}^n(f^{m_0} (Y_{i})-f^{m_0} (X_i))\right)\to \int f^{m_0} \, \dd\mu_2 -\int f^{m_0} \, \dd\mu_1\geq 1-2/{m_0},\quad n\to \infty,
$$
which is a contradiction. Therefore $\mu_1=\mu_2$ follows.
\end{proof}
In the proof Theorem \ref{conv2} we will use the following proposition, whose proof is postponed to Appendix B.
\begin{proposition}\label{jan_cor} Under the assumptions of Theorem \ref{conv2},
there exists a measurable mapping
$$
M\ni y\mapsto \xi_{x,y} \in \mathcal{P}(E^\infty\times E^\infty)
$$
such that for $\mu$-a.a. $y\in M$ one has $\xi_{x,y} \in \widehat C(\P_x,\P_y)$, $\pi_1(\xi_{x,y})\sim \P_x$, and \eqref{prob1} holds true.
\end{proposition}

\begin{proof}[Proof of Theorem \ref{conv2}] Define
the measure $\xi\in \mathcal{P}(E^\infty\times E^\infty)$  as follows:
$$
\xi(A)={1\over \mu(M)}\int_{M}\xi_{x,y}(A)\mu(dy), \quad A\in\mathcal{E}^{\otimes\infty}\otimes \mathcal{E}^{\otimes\infty},
$$
where $\{\xi_{x,y}, y\in M\}$ is defined in Proposition \ref{jan_cor}.
Because $\pi_1(\xi_{x,y})\sim \P_x, \pi_2(\xi_{x,y})\ll \P_\mu$ for $\mu$-a.a. $y\in M$, we have that
 $\xi\in \widehat C(\P_x,\P_\mu)$ and $\pi_1(\xi)\sim \P_x$. In addition, we have
 \begin{equation}\label{prob11}
\lim_{n \to \infty} \xi \big( d(X_n,Y_n) \le \varepsilon\big) =1
\end{equation}
for every $\varepsilon>0$.

Denote
$$
\mu_n(C)=\xi(X_n\in C), \quad \nu_n(C)=\xi(Y_n\in C), \quad  C\in\mathcal{E}.
$$
Observe first that the family $\{\nu_n, n\geq 1\}\subset \mathcal{P}(E)$ is tight: recall that  $\pi_2(\xi)\ll \P_\mu$ and thus for every $\eps>0$ there exists $\delta>0$ such that for any
$B\in  \mathcal{E}^{\otimes \infty}$ with $\P_\mu(B)\leq \delta$
one has $\pi_2(\xi)\big(B\big)\leq \eps$. Therefore
$$
\nu_n(K)=\xi\big(Y_n\in K\big)\leq \eps, \quad n\geq 1
$$
if a compact set $K\subset E$ is chosen such that $\mu(K)\geq 1-\delta.$

Since the embedding $(E,\rho)$ into $(\overline{E}^d,d)$ is continuous, this yields that $\{\nu_n, n\geq 1\}\subset \mathcal{P}(\overline{E}^d)$ is tight, as well. Then using using (\ref{prob11}), we deduce that $\{\nu_n, n\geq 1\}\subset \mathcal{P}(\overline{E}^d)$ is tight, as well. Because $\P_x\ll \pi_1(\xi)$, this finally yields the tightness of $\{P_n(x,.), n\geq 1\}\subset \mathcal{P}(\overline{E}^d)$.

The metric space $(\overline{E}^d,d)$ is complete by the construction and is separable since $(E,\rho)$ is separable. Hence if we assume that
  that $P_n(x,.), n\geq 1$ does not weakly converge to $\mu$ w.r.t. $d$, then  there exists some probability measure $\nu \neq \mu$ on $\overline{E}^d$ and a subsequence
$P_{m_k}(x,.) \Tod\nu$. Fix a bounded $d$-Lipschitz continuous function $f:\overline{E}^d\to \R$ such that $\bar f:=\int_E f\,\dd \mu \neq \int_{\overline{E}^d} f\,\dd \nu$ and put
$$
U_n=\frac 1n \sum_{k=1}^n f(X_{m_k}).
$$
Recall that the chain $X$ is stationary and mixing w.r.t. $\P_{\mu}$, hence
$$
\mathrm{Cov}_\mu(f(X_m), f(X_n)):=\E_\mu  \Big( f(X_m)f(X_n)\Big) -(\bar f)^2\to 0, \quad |n-m|\to \infty.
$$
Then
$$
\E_{\mu}(U_n-\bar f)^2=\frac 1{n^2}\sum_{i=1}^n\sum_{j=1}^n  \mathrm{Cov}_\mu(f(X_{m_i}), f(X_{m_j}))\to 0, \quad n\to \infty,
%
$$
and furthermore there exists a sequence $\{n_r\}$ such that $U_{n_r}\to \bar f, r\to \infty$ a.s. w.r.t. $\P_\mu$. Because $\pi_2(\xi)\ll \P_\mu$, we have finally that
$U_{n_r}\to \bar f, r\to \infty$ a.s. with respect to
$\pi_2(\xi)$.

Recall that $f$ is Lipschitz continuous,
then by (\ref{prob11}) the sequence
$$
\Delta_m:=|f(X_m)-f(Y_m)| \leq \mathrm{Lip}(f)d(X_m,Y_m), \quad m\geq 1
$$
converges to $0$ in $\xi$-probability. Because $f$ is bounded, this convergence  holds true also in the mean sense, which then implies
$$
\frac 1n \sum_{k=1}^n f(X_{m_k})-\frac 1n \sum_{k=1}^n f(Y_{m_k})\to 0, \quad n\to \infty
$$
in the mean sense w.r.t. $\xi$. Since $\{Y_n\}$ has the law $\pi_2(\xi)$ w.r.t. $\xi$ and $U_{n_r}\to \bar f, r\to \infty$ a.s. w.r.t. $\pi_2(\xi)$, we have then
$$
\Big( \frac 1{n_r} \sum_{k=1}^{n_r} f(X_{m_k}), \quad \frac 1{n_r} \sum_{k=1}^{n_r}f(Y_{m_k})\Big)\to (\bar f,\bar f), \quad r\to \infty
$$
in $\xi$-probability,
 hence $U_{n_r}\to \bar f, r\to \infty$ in probability with respect to $\pi_1(\xi)$. Because $\P_x\ll \pi_1(\xi)$, we deduce finally that
  $U_{n_r}\to \bar f, r\to \infty$ in probability with respect to  $\P_x$. Since $f$ is bounded, the sequence $U_{n_r}, r\geq 1$ is bounded by the same constant, and then
  we have
  $$
  \E_x  U_{n_r}\to \bar f, \quad r\to \infty.
  $$
On the other hand, by the assumption $P_{m_k}(x,.) \Tod \nu$, we have $\int_{\overline{E}^d} f(z)P_{m_k}(x,\dd z)\to \int_{\overline{E}^d} f\dd\nu$ and therefore
$$\E_x\frac 1n \sum_{k=1}^n f(X_{m_k})=\frac 1n \sum_{k=1}^n\int_E f(z)P_{m_k}(x,\dd z)=\frac 1n \sum_{k=1}^n\int_{\overline{E}^d} f(z)P_{m_k}(x,\dd z)\to \int_E f \,\dd \nu \neq \bar f,\quad n\to \infty,$$ which is a contradiction
finishing the proof of the theorem.
\end{proof}

\section{Proofs of Theorem \ref{conv1}, Theorem \ref{conv11}, and Corollary \ref{coromixing}}\label{s4} In this section
we prove Theorem \ref{conv1}, which is the most complicated of our  main results. We also  show how (parts of) the proof can be modified in order to obtain
Theorem \ref{conv11}, and prove  Corollary \ref{coromixing}.

Before proceeding with these proofs,  we  formulate several auxiliary results.

\begin{proposition}\label{newlemma} I. Let  $p >1$ and $R>0$ be fixed. Then for each $\alpha>0$ there exists some $\alpha'>0$ such that the following holds true:
for every $\P,\Q\in \mathcal{P}(E^\infty)$ and every  $\xi\in \widehat  C^R_p(\P,\Q)$  there exists some $\zeta \in C(\P,\Q)$ such that for each
$A \in \EE^{\otimes\infty}  \otimes \EE^{\otimes\infty}$
satisfying  $\xi(A)\ge \alpha$
we have $\zeta(A)\ge \alpha'$.

II. For each $\alpha>0$ there exists some $\alpha'>0$ and $R\geq 1$ such that the following holds true:
for every $p\geq 1$, every $\P,\Q\in \mathcal{P}(E^\infty)$ and every  $\xi\in \widehat  C(\P,\Q)$  there exists some $\zeta \in \widehat C^R_p(\P,\Q)$ such that for each
$A \in \EE^{\otimes\infty}  \otimes \EE^{\otimes\infty}$
satisfying  $\xi(A)\ge \alpha$
we have $\zeta(A)\ge \alpha'$.
\end{proposition}

The proof of this proposition is given in Appendix A.

\begin{proposition}\label{prop} Let $S_1, S_2$ be Polish spaces and $Q:S_1\to \mathcal{P}(S_2)$ be a continuous mapping. Let $h:S_2\times S_2\to [0,1]$ be a distance-like
function.

Then there exists a measurable mapping $\eta:S_1\times S_1\ni (x,y)\mapsto \eta_{x,y}\in \mathcal{P}(S_2\times S_2)$ such that
$$
\eta_{x,y}\in C(Q(x), Q(y)), \quad
\int_{S_2\times S_2}h(x', y')\eta_{x,y}(\dd x', \dd y')=h(Q(x), Q(y)), \quad x,y\in S_1.
$$
\end{proposition}

The proof of this proposition is given in Appendix C.

\begin{corollary}\label{measur_coup} For a given $d$-Feller chain $X$ and  given $n\in \mathbb{N}, \eps>0$, denote
$$
\gamma_{x,y}^{n,\eps}:=\sup_{\xi \in C(\P_x,\P_y)}\xi(d(X_n,Y_n)\le \varepsilon), \quad (x,y)\in E\times E.
$$
The following statements hold.

 I. The function $\gamma_{\cdot}^{n,\eps}:E\times E\to [0,1]$ is $\EE \otimes \EE-\mathcal{B}([0,1])$ measurable.

II. There exists a measurable function
$$
\xi^{n,\eps}:E\times E\ni(x,y)\mapsto \xi^{n,\eps}_{x,y}\in \mathcal{P}(E^\infty\times E^\infty)
$$
 such that for every $(x,y)\in E\times E$ the following properties hold:
\begin{itemize}
  \item[(i)] $\xi^{n,\eps}_{x,y}\in  C(\P_x,\P_y)$;
  \item[(ii)]
  $$
  \xi^{n,\eps}_{x,y}(d(X_n,Y_n)\le \varepsilon)=\gamma_{x,y}^{n,\eps}.
  $$
 \end{itemize}
\end{corollary}

\begin{proof}  Consider the Polish metric space $(\overline{E}^d, d)$. Consider also  the space $(\overline{E}^d)^\infty$ with the metric $d^{(\infty)}$ introduced in the
same way with the metric $\rho^{(\infty)}$, and the space $\mathcal{P}((\overline{E}^d)^\infty)$ with the Kantorovich-Rubinshtein distance $d^{(\infty)}$; see Section \ref{basic}.
By \eqref{approx} we have that the image of $E^\infty$ under the natural embedding is a measurable subset in $(\overline{E}^d)^\infty$, and $\mathcal{P}(E^\infty)$ can be
identified as the (measurable) set of those measures from $\mathcal{P}((\overline{E}^d)^\infty)$ which provide a full measure for $E^\infty$. The same remarks are valid for
the spaces  $(\overline{E}^d)^\infty\times (\overline{E}^d)^\infty$ and $\mathcal{P}((\overline{E}^d)^\infty\times (\overline{E}^d)^\infty)$ which are defined analogously.

Observe that, because the chain $X$ is $d$-Feller, the mapping
$$
E\ni x\mapsto \P_x\in \mathcal{P}((\overline{E}^d)^\infty)
$$
is continuous. Hence we can apply Proposition \ref{prop} with  $S_1=E, S_2=(\overline{E}^d)^\infty$, $Q(x)=\P_x, x\in E$, and
$$
h(x,y)=\1_{d(x_n, y_n)> \eps} =1-\1_{d(x_n, y_n)\leq \eps}, \quad  x=(x_k)_{k\geq 1}, \, y=(y_k)_{k\geq 1}\in (\overline{E}^d)^\infty.
$$
Then  there exists a measurable function $$
\xi^{n,\eps}:E\times E\ni(x,y)\mapsto \xi^{n,\eps}_{x,y}\in \mathcal{P}((\overline{E}^d)^\infty\times (\overline{E}^d)^\infty)
$$
which satisfies properties (i), (ii) in  statement II of the corollary. In addition,  each $\P_x, x\in E$ assigns full measure to $E^\infty$, hence
by the property (i) each measure $\xi^{n,\eps}_{x,y}, (x,y)\in E\times E$ assigns full measure to $E^\infty\times E^\infty$. Therefore $\xi^{n,\eps}$ can be considered
as a measurable mapping taking values in $\mathcal{P}(E^\infty\times E^\infty)$, which completes the proof of  statement II.

Statement I follows immediately, because the mapping
$$
(x,y)\mapsto \xi^{n,\eps}_{x,y}(d(X_n,Y_n)\le \varepsilon)
$$ is measurable.

\end{proof}

\begin{remark} Proposition \ref{prop} and Corollary \ref{measur_coup} give a natural extension  of the ``Coupling Lemma for transition probabilities''
(Lemma 1 in \cite{KS15}). This lemma provides a probability kernel which in a point-wise sense minimizes the particular distance-like function
$h(x,y)=1_{x\not=y}$, while Proposition \ref{prop} provides such a kernel for an arbitrary distance-like function. The proof of Lemma 1 in \cite{KS15}
exploits an explicit construction of a maximal coupling based on the splitting representation of a probability law, and it can not be extended to our current setting.
We use instead the general measurable selection theorem which dates back to Kuratovskii and Ryll-Nardzevski theorem combined with some measurability criteria for
set-valued maps explained in \cite{Stroock_Varad}, Chapter 12.1. We mention that our proof is similar to that of Lemma 4.13 in \cite{HMS11},
which also provides a probability kernel which is maximal w.r.t. $h$, but in our setting we avoid using an additional
assumption on $h$ to be continuous.
\end{remark}

\begin{remark}\label{r45} We mention for future reference that the  kernel $\xi^{n, \eps}$ can be modified such that it possesses the following additional property,
which is a direct analogue of  property (ii) of the maximal coupling kernel constructed in \cite[Lemma 1]{KS15}:
\begin{itemize}
  \item[(iii)] the measure $\xi^{n,\eps}_{x,y}$ conditioned by $\{d(X_n,Y_n)> \varepsilon\}$ is absolutely continuous w.r.t.
  $\P_x\otimes \P_y$ with the respective Radon-Nikodym density being bounded from above by $(1-\gamma_{x,y}^{n,\eps})^{-1}$.
\end{itemize}
Namely, denote by  $\eta^{n,\eps}_{x,y}$ and $\zeta^{n,\eps}_{x,y}$ the initial measure $\xi^{n,\eps}_{x,y}$ conditioned by  the event
$\{d(X_n,Y_n)\leq  \varepsilon\}$ and by its complement, respectively. Then it is easy to see that the modified function
$$
\gamma^{n,\eps}_{x,y}\eta^{n,\eps}_{x,y}+(1-\gamma^{n,\eps}_{x,y})\pi_1(\zeta^{n,\eps}_{x,y})\otimes \pi_1(\zeta^{n,\eps}_{x,y})
$$
satisfies (i) -- (iii); see the proof of Theorem \ref{conv1} below for a more detailed discussion of this construction in a slightly different setting.
\end{remark}

\begin{proposition}\label{singular}
If $\nu_1$ and $\nu_2$ are singular probability measures on a Polish space $(E,\rho)$, then
$$
\lim_{\varepsilon \downarrow 0} \sup_{\zeta \in C(\nu_1,\nu_2)}\zeta \big( d(X,Y)\le \varepsilon\big) =0.
$$
\end{proposition}

The proof of this proposition is given in Appendix A.

\begin{corollary}\label{ergomu}
Under the conditions of  Theorem \ref{conv1}, the measure $\mu$ is ergodic.
\end{corollary}
\begin{proof} Consider the ergodic decomposition \eqref{decomp} of $\mu$. Then
$$
1=(\mu\otimes\mu)(M)=\int_{E\times E}(\nu_1\otimes \nu_2)(M)\kappa(\dd \nu_1)\kappa(\dd \nu_2).
$$
If $\mu$ is not ergodic then $\kappa$ is non-degenerate and there exist two mutually singular invariant probability measures $\nu_1, \nu_2$  such that
\begin{equation}\label{ass}
(\nu_1\otimes \nu_2)(M)=1.
\end{equation}
Define for a given $n\geq 1, \eps>0$ the measure $\eta^{n,\eps}\in  \mathcal{P}(E^\infty\times E^\infty)$ by
$$
\eta^{n, \eps}=\int_{E\times E} \xi_{x,y}^{n, \eps}\,\nu_1(\dd x)\nu_2(\dd y),
$$
where $\xi_{x,y}^{n, \eps}$ is defined as in Corollary \ref{measur_coup},
and denote by $\zeta^{n, \eps}$ the law of $(X_n, Y_n)$ under $\eta^{n, \eps}$. Because  $\pi_1(\eta^{n, \eps})=\P_{\nu_1}, \pi_2(\eta^{n, \eps})=\P_{\nu_2}$, and
$\nu_1, \nu_2$ are invariant, we have
$$
\zeta^{n, \eps}\in C(\nu_1,\nu_2)
$$
for any $n\geq 1, \eps>0$.

On the other hand,
$$
\zeta^{n, \eps}( d(X,Y)\le \varepsilon)= \int_{E\times E}\gamma_{x,y}^{n, \eps}\,\nu_1(\dd x)\nu_2(\dd y)
$$
and therefore
$$
\sup_{\zeta \in C(\nu_1,\nu_2)}\zeta \big( d(X,Y)\le \varepsilon\big)\geq \int_{E\times E}\gamma_{x,y}^{n, \eps}\,\nu_1(\dd x)\nu_2(\dd y).
$$
Denote
$$
\gamma_{x,y}^\eps= \liminf_{n \to \infty} \gamma_{x,y}^{n,\eps},
\quad
\gamma_{x,y}= \lim_{\eps\to 0+} \gamma_{x,y}^{\eps},
$$
then by the Fatou lemma and the monotone convergence theorem
$$
\lim_{\varepsilon \downarrow 0} \sup_{\zeta \in C(\nu_1,\nu_2)}\zeta \big( d(X,Y)\le \varepsilon\big)\geq \int_{E\times E}\gamma_{x,y}\,\nu_1(\dd x)\nu_2(\dd y).
$$
By condition (i) of Theorem \ref{conv1}, we have $\gamma_{x,y}>0$ for any $(x,y)\in M$, hence the above inequality combined with \eqref{ass} contradicts Proposition
\ref{singular}.
\end{proof}

\begin{corollary}\label{ergomu2}
Under the conditions of  Theorem \ref{conv1}, the measure $\mu\otimes \mu$ is ergodic for the product chain.
\end{corollary}
\begin{proof} Denote by the same symbol $d$ the metric on the product space $E\times E$ $$
d\big((x,u), (y,v)\big)=d(x,y)\wedge d(u,v),
$$
and by $\P_{(x,u)}$  the distribution of the product chain with the initial value $(x,u)$. For any $x,y,u,v\in E$ and $n\geq 1, \eps>0$ consider the probability measure
on $E^\infty\times E^\infty\times E^\infty\times E^\infty$
$$
\xi^{n,\eps}_{x,y,u,v}=\xi^{n,\eps}_{x,y}\otimes \xi^{n,\eps}_{u,v},
$$
where $\xi^{n,\eps}_{x,y}$ is defined in Corollary \ref{measur_coup}. Then the projections $\pi_{1,3}$ and $\pi_{2,4}$ of this measure on the coordinates 1,2 and 2,4,
respectively equal $\P_{(x,u)}$ and $\P_{(y,v)}$. On the other hand,
$$
\lim_{\varepsilon \downarrow 0+}\liminf_{n \to \infty}  \xi^{n,\eps}_{x,y,u,v} \Big( d\big((X_n,U_n),(Y_n,V_n)\big) \le \varepsilon \Big)=
\lim_{\varepsilon \downarrow 0+}\liminf_{n \to \infty} \gamma^{n,\eps}_{x,y}\gamma^{n,\eps}_{u,v}\geq \gamma_{x,y}\gamma_{u,v}>0
$$
for any
$$
(x,y,u,v)\in M':=M\times M.
$$
 Thus the above inequality yields the following analogue of
 condition (i) of Theorem \ref{conv1} for the product chain: for any $(x,y,u,v)\in M'$,
\begin{equation}\label{prod}\lim_{\varepsilon \downarrow 0+}\liminf_{n \to \infty} \sup_{\pi_{1,3}(\xi)=\P_{(x,u)},\pi_{2,4}(\xi)=\P_{(y,v)}}
\xi\Big( d\big((X_n,U_n),(Y_n,V_n)\big) \le \varepsilon \Big) >0.
\end{equation}

Clearly,  $(\mu\otimes \mu\otimes \mu\otimes \mu)(M')=1$, hence  the required statement follows by the previous corollary.
\end{proof}

Now we proceed with the proof of Theorem \ref{conv1}. The second statement of the theorem follows easily: since
$$
C(\P,\Q) \subset \widehat  C_p^R(\P,\Q)
$$
 for each $p>1$ and $R \ge 1$, condition (i) immediately implies (ii). The inverse implication follows from the first statement in Proposition \ref{newlemma}
 while the second statement in this proposition shows that (iii) implies (ii).

\begin{proof}[Proof of Theorem \ref{conv1}, statement I]
Our aim is to prove that for every $\eps>0$
\begin{equation}\label{weak_in_prob2}
\Gamma^{n,\eps}:=\int_{E\times E}\,\gamma_{x,y}^{n, \eps}\mu(\dd x)\mu(\dd y)\to 1, \quad n\to \infty.
\end{equation}
This yields the required statement. Indeed, for given $\eps>0$ and $n \ge 1$, consider a random element $\eta$ with law $\mu$ and a sequence $Z_k=(X_k,Y_k), k\geq 0$
with $Z_0=(x,\eta)$ and the conditional law of $Z$ under $\sigma(Z_0)$ equal  $\xi^{n,\eps}_{x,\eta}$;  the measurable mapping $\xi^{n,\eps}$ is introduced in
Corollary \ref{measur_coup}. By  property (i) of this mapping, the law of $Z$ belongs to
$C(\P_x, \P_\mu)$. We have assumed $d\leq 1$,  hence
$$
d\Big(P_n(x,\cdot),\mu\Big)\leq \eps+\xi^{n,\eps}_{x,\eta}   \Big(d(X_n, Y_n)> \eps\Big).
$$
By property (ii) of the mapping $\xi^{n,\eps}$, we have
$$
\xi^{n,\eps}_{x,\eta}\Big(d(X_n, Y_n)\leq \eps\Big)=\int_{E}\,\gamma_{x,y}^{n, \eps}\mu(\dd y).
$$
Consequently, it follows from \eqref{weak_in_prob2} that
$$
\liminf_{n\to \infty}\int_{E}d\Big(P_n(x,\cdot),\mu\Big)\, \mu(\dd x)\leq \liminf_{n\to \infty}\left(\eps+(1-\Gamma^{n,\eps})\right)=\eps
$$
for any $\eps>0$, which then yields  \eqref{weak_in_prob}.\\

Now we proceed with the proof of  \eqref{weak_in_prob2}. Take an independent coupling $Z_k=(X_k,Y_k),\,k\in \Z_+$ with $\mathrm{Law}\,(X_0)=\mathrm{Law}\,(Y_0)=\mu$ on a probability
space $(\Omega, {\mathcal{F}}, \P)$ and observe that for every fixed $\eps, n, k$
\begin{equation}\label{submart}
\gamma_{Z_k}^{n+1,\eps}\geq \E[ \gamma_{Z_{k+1}}^{n,\eps}|\mathcal{F}^Z_k].
\end{equation}
Indeed,  the expression on the left hand side means that one fixes the position of $Z$ at the time instant $k$ and optimizes the probability for the coordinates of $Z$
to stay $\eps$-close at the time instant $n+k+1$, while the expression at the right hand side means that
one makes an independent step first, and then optimizes the same the probability; the optimal probability in the second case is smaller because due to the more restricted
set of possible couplings.

Using Fatou's lemma and inequality \eqref{submart} we obtain:
$$
 \E[\gamma_{Z_{k+1}}^{\eps}|\mathcal{F}^Z_k]=
 \E[\liminf_n\ \gamma_{Z_{k+1}}^{n,\eps}|\mathcal{F}^Z_k]\leq  \liminf_n \E[\gamma_{Z_{k+1}}^{n,\eps}|\mathcal{F}^Z_k]\leq \liminf_n \gamma_{Z_k}^{n+1,\eps}=
 \gamma_{Z_k}^{\eps},
$$
where $\gamma^\eps, \gamma$ are defined as in the proof of Corollary \ref{ergomu}. Hence $\gamma_{Z_n}^{\eps}, n\geq 1$ is a non-negative super-martingale for every
$\eps>0$, and so is $\gamma_{Z_n}, n\geq 1$. Therefore, the $\P^Z$-a.s. limits
$$
  \gamma_{Z_n}\to \gamma, \quad \gamma_{Z_n}^{\eps}\to \gamma^{\eps}, \quad n\to \infty
  $$
exist. On the other hand, since $Z$ is stationary, the sequences    $\gamma_{Z_n}^{\eps}, n\geq 1$ and  $\gamma_{Z_n}, n\geq 1$ are stationary as well, and thus
each $\gamma_{Z_n}$ (resp. $\gamma_{Z_n}^{\eps}$) has the same law as $\gamma$  (resp. $\gamma^{\eps}$). By Corollary \ref{ergomu2},
the process $Z$ is ergodic and therefore Birkhoff's ergodic theorem implies
$$
\gamma=\lim_{n\to \infty}{1\over n}\sum_{k=1}^n\gamma_{Z_k}, \quad \gamma^{\eps}=\lim_{n\to \infty}{1\over n}\sum_{k=1}^n\gamma_{Z_k}^{\eps}
$$
are almost surely constant. We can therefore assume that $\gamma$ and $\gamma^{\eps}$ are deterministic. It follows that
$$
\gamma^\eps_{x,y}=\gamma^\eps, \quad \gamma_{x,y}=\gamma
$$
for $\mu\otimes\mu$-a.a. $(x,y)\in E\times E$. Observe that $\gamma^\eps\geq \gamma$, and by assumption (i) of the theorem we have $\gamma>0$.

The same reasoning as the one we have used to prove \eqref{submart} shows that
$$
\Gamma^{n+1,\eps}\geq \Gamma^{n, \eps},
$$
and clearly $\eps \mapsto \Gamma^{n, \eps}$ is non-decreasing. Hence there exist the limits
$$
\Gamma^\eps=\lim_{n\to \infty}\Gamma^{n,\eps}, \quad \Gamma=\lim_{\eps\to 0}\Gamma^\eps.
$$
To show \eqref{weak_in_prob2}, we just need to show that $\Gamma=1$. Observe that
$\Gamma^{n,\eps}$ equals the maximal probability of the  event $\{d(X_n, Y_n)\leq \eps\}$ over all couplings $Z=(X,Y)$ such that $\mathrm{Law}(Z_0)=\mu\otimes\mu$ and
the conditional distributions of $X,Y$ conditioned by $\sigma(Z_0)$ equal $\P_{X_0}, \P_{Y_0}$, respectively.

Assuming $\Gamma<1$, we will construct for any fixed $\eps>0$ and any $n$ large enough  a  coupling $Z=(X,Y)$ having the same properties as above  such that
\begin{equation}\label{terminal}
\P(d(X_n, Y_n)\leq \eps)\geq \Gamma+{(1-\Gamma)\gamma\over 2}.
\end{equation}
This  yields the contradictory inequality
$$
\Gamma=\lim_{\eps\to 0}\lim_{n\to \infty}\Gamma^{n, \eps}\geq \Gamma+{(1-\Gamma)\gamma\over 2}>\Gamma,
$$
and hence $\Gamma<1$ is impossible. Note however that our proof does not show that $\gamma=1$.

Fix  $\gamma'\in (\gamma/2, \gamma)$ and $\Gamma'\in (0, \Gamma)$ close enough to $\Gamma$, so that $$
\Gamma'+(1-\Gamma')\gamma'> \Gamma+{(1-\Gamma)\gamma\over 2}.
$$
Then choose $\delta>0$ small enough, so that
\begin{equation}\label{delta}
(1-\delta)(\Gamma'-2\delta)+\gamma'(1-\Gamma')-\delta> \Gamma+{(1-\Gamma)\gamma\over 2}.
\end{equation}

Let us proceed with a preliminary analysis which will give us several auxiliary objects we will use in the construction below. First,
 since $\gamma_{x,y}^\eps=\gamma^\eps\geq \gamma$ for $\mu\otimes\mu$-a.a. $(x,y)\in E\times E$, by the definition of $\gamma_{x,y}^\eps$  we have that the time moment
$$
T_{x,y}^{\eps}=\min\{T: \gamma_{x,y}^{n,\eps}>\gamma', \quad n\geq T \}
$$
is finite for $\mu\otimes\mu$-a.a. $(x,y)\in E\times E$. Fix some $N$ such that
$$
(\mu\otimes\mu)\Big((x,y):T_{x,y}^\eps>N\Big)<\delta,
$$
and denote $O_N^\eps=\{(x,y):T_{x,y}^\eps\leq N\}$.

Next, recall that  $X$ is assumed to be $d$-Feller, and $d$ is continuous. Then  for  given $\eps>0$, ащк $N\geq 1$ chosen above,
and ащк any compact set $K\subset E$
$$
\gamma_{x,y}^{N, \eps}\to 1
$$
when $d(x,y)\to 0, (x,y)\in K\times K$. Fix a compact set $K\subset E$ such that $\mu(K)>1-\delta,$ and choose $\eps_1>0$ such that
$$
\gamma_{x,y}^{N, \eps}>1-\delta, \quad d(x,y)\leq \eps_1, \quad (x,y)\in K\times K.
$$

Finally,  we observe that by the definition of $\Gamma^{\eps_1}\geq \Gamma>\Gamma'$, there exists $N_0\in  \N$ such that $\Gamma^{\eps_1, N_0}\geq\Gamma'$.
Hence for arbitrary $n\geq N_0$ there exists a coupling $Z^{n,\eps_1}=(X^{n,\eps_1},Y^{n,\eps_1})$ such that
$\mathrm{Law}(Z_0^{n,\eps_1})=\mu\otimes\mu$, the conditional distributions of $X^{n,\eps_1},Y^{n,\eps_1}$ conditioned by $\sigma(Z_0^{n,\eps_1})$ equal
$\P_{X_0^{n,\eps_1}}, \P_{Y_0^{n,\eps_1}}$ respectively, and
\begin{equation}\label{ineq}
\P\Big(d(X_{n}^{n,\eps_1}, Y_{n}^{n,\eps_1})\leq \eps_1\Big)\geq \Gamma'.
\end{equation}
We modify this coupling by the  same construction we have mentioned in Remark \ref{r45}. Namely, denote the law of $Z$ by $\P^Z$ and consider the set
$$
C=\Big\{d(X_{n}^{n,\eps_1}, Y_{n}^{n,\eps_1})\leq \eps_1\Big\}.
$$
Then
$$
\P^Z=\Gamma' \P^Z(\cdot|C)+(1-\Gamma')\Q^{Z, \Gamma', C},
$$
where
$$
\Q^{Z, \Gamma', C}=(1-\Gamma')^{-1}(\P^Z-\Gamma' \P^Z(\cdot|C))
$$
is a probability measure on $E^\infty\times E^\infty$. Recall that the projections $\pi_1, \pi_2$ of $\P^Z$ equal $\P_\mu$, hence the projections of $\Q^{Z, \Gamma', C}$
are absolutely continuous w.r.t. $\P_\mu$ with their Radon-Nikodym derivatives $\leq  (1-\Gamma')^{-1}$. Taking instead of $\P^Z$ the measure
$$
\Gamma' \P^Z(\cdot|C)+(1-\Gamma')\pi_1(\Q^{Z, \Gamma', C})\otimes \pi_2(\Q^{Z, \Gamma', C}),
$$
we obtain a new coupling such that \eqref{ineq} for this coupling still holds true, but in addition the  distribution conditioned by the complement to the set $C$ is
absolutely continuous w.r.t. $\P_\mu\otimes\P_\mu$ with the Radon-Nikodym density bounded by $(1-\Gamma')^{-1}$. With a slight abuse of notation which however does not cause
misunderstanding, we denote this modified coupling  by the same symbol $Z^{n, \eps_1}$.

Now for an arbitrary $n\geq N_0+N$ we construct the required coupling $Z$ such that \eqref{terminal} holds true.
We define $Z$ as follows:
\begin{itemize}
 \item the law of $Z_k, k\leq n-N$ is the same as the law of $Z_k^{n-N,\eps_1}, k\leq n-N$ (recall that $n-N\geq N_0$ hence $Z^{n-N,\eps_1}$ is well defined);
 \item the conditional law of $Z_{l+n-N}, l\geq 0$ w.r.t. $\sigma(Z_k, k\leq n-N)$ equals $\xi^{N, \eps}_{X_{n-N},Y_{n-N}}$, where $\xi^{n,\eps}$ is the function
 constructed in Lemma \ref{measur_coup}.
\end{itemize}

To estimate  the probability of the event $A=\{d(X_{n},Y_{n})\leq \eps\}$,  denote
$$
B=\{d(X_{n-N},Y_{n-N})\leq \eps_1\}, \quad C=\{Z_{n-N}\in K\times K\}, \quad D=\{Z_{n-N}\in O_{N}^{\eps}\}.
$$
Observe that, when conditioned by $B\cap C$, the event $A$  has probability $\geq 1-\delta$ because the components $X,Y$ start $\eps_1$-close from the compact set $K$
and hence  by the choice of $N$ they stay $\eps$-close
after the time $N$ with probability $\geq 1-\delta$.

On the other hand, when conditioned by $\overline B\cap D$,  event $A$  has probability $\geq \gamma'$ by the definitions of the set $O^\eps_N$ and the event $D$,
and according to our construction of
the coupling $Z$. Therefore
$$
\P(A)\geq (1-\delta)\P(B\cap C)+\gamma'\P(\overline B\cap D).
$$
Recall that each of the components $X,Y$ has  law $\P_\mu$, hence
$$
\P(B\cap C)\geq \P(B)-\P(X_{n-N}\not\in K)-\P(Y_{n-N}\not\in K)\geq \P(B)-2\delta.
$$
Next, the  law of $Z_{n-N}$ conditioned by $\overline B$ is absolutely continuous w.r.t. $\mu\otimes \mu$ with Radon-Nikodym density $\leq (1-\Gamma')^{-1}$. Hence
$$
\P(\overline D|\overline B)\leq (1-\Gamma')^{-1}(\mu\otimes \mu)((E\times E)\setminus O_{N}^{\eps})\leq (1-\Gamma')^{-1}\delta
$$
and therefore
$$
\P(\overline B\cap D)=\P(D|\overline B)(1-\P(B))\geq \Big(1- (1-\Gamma')^{-1}\delta\Big)(1-\P(B)).
$$
Recall that $\P(B)=\Gamma'$, so we finally obtain
$$
\P(A)\geq (1-\delta)(\Gamma'-2\delta)+\gamma'(1-\Gamma')-\delta.
$$
By \eqref{delta}, this yields \eqref{terminal} and completes the proof.
\end{proof}

\begin{proof}[Proof of Theorem \ref{conv11}] Like in the  previous proof,
it is sufficient to show that for any $\eps>0$ the constant $\gamma^\eps$ constructed above  equals 1. Fix $x_0$ in the (topological) support of $\mu$ and observe
that by the e-chain property and the
triangle inequality for the metric $d$ on $\mathcal{P}(E)$, for any $\kappa>0$ there exists $r>0$ such that
\begin{equation}\label{21}
d(P_n(x,\cdot), P_n(y,\cdot))\leq \kappa, \quad n\geq 0, \quad x,y\in B(x_0, r),
\end{equation}
where $B_d(x_0, r)$ is the open ball in $E$ w.r.t. $d$ with center $x_0$ and radius $r$.  Note that for any $\eps>0$ and any coupling $\xi\in C(\P_x, \P_y)$,
\begin{equation}\label{22}
\E^\xi d(X_n, Y_n)\geq  \eps\xi(d(X_n, Y_n)> \eps)\geq \eps(1-\gamma_{x,y}^{n,\eps}),
\end{equation}
hence
$$
\gamma_{x,y}^{n,\eps}\geq 1-{1\over \eps}\E^\xi d(X_n, Y_n)
$$
Since
$$
d\Big(P_n(x,\cdot), P_n(y,\cdot)\Big)=\min_{\xi\in C(\P_x, \P_y)}\E^\xi d(X_n, Y_n),
$$
combining \eqref{21} and \eqref{22} we get
$$
\gamma_{x,y}^{n,\eps}\geq 1-{\kappa\over \eps}, \quad n\geq 0, \quad x,y\in B(x_0, r).
$$
Because $B(x_0, r)\times B(x_0, r)$ has positive measure $\mu\otimes \mu$ for any $r>0$, and
$$
\liminf_{n\to \infty}\gamma^{n, \eps}_{x,y}=\gamma^\eps
$$
for $\mu\otimes \mu$-a.a. $(x,y)$, the above inequality yields
$$
\gamma^\eps\geq 1-{\kappa\over \eps}
$$
for any $\kappa>0$; that is, $\gamma^\eps=1$.
\end{proof}

\begin{proof}[Proof of Corollary \ref{coromixing}] Let $g:E \to {\mathbb{R}}$ be Lipschitz continuous w.r.t. $d$.
Then
$$
\Big|\E_\mu[g(X_n)|X_j, j\leq 0]-\E_\mu g(X_0)\Big|\leq \mathrm{Lip}(g)  d( P_n(X_0,\cdot), \mu).
$$
Since $\mathrm{Law}\,(X_0)=\mu$,  \eqref{mixing} follows from \eqref{weak_in_prob} by the dominated convergence theorem (recall that we assume $d\leq 1$).

For an arbitrary bounded $g$,  the usual approximation arguments can be applied since the time shift is an isometry on $L_2(E^\infty, \P_\mu)$ and the class of
$d$-Lipschitz continuous functions is
dense in $L_2(E, \mu)$.
\end{proof}

\section{Examples}\label{s5}

In this section we give several examples which illustrate the conditions imposed in our main results and clarify the relations of these results
with some other available in the field.

\begin{example}\label{examplelimsup}
This example shows that assumption \eqref{limsup} in Theorem \ref{unique} cannot be replaced by the assumption
\begin{equation}\label{limsupmod}
\limsup_{n \to \infty} \xi_{x,y} \big( d(X_n,Y_n)\le \varepsilon\big) =1,
\end{equation}
even if the chain is Feller and
generalized couplings are replaced by couplings. Consider the torus $E=[0,1)$ equipped with the Euclidean metric $d(x,y)=|y-x|\wedge (1-|y-x|)$ and consider the
deterministic map $x \mapsto 2x$ mod 1, $\mu_1=\delta_0$, $\mu_2=\lambda$, where $\lambda$ is the Lebesgue measure on $E$. Both $\delta_0$ and $\lambda$ are invariant and
ergodic and for $\lambda$-almost all $y \in E$ there exists a (deterministic) sequence along which the transition probabilities starting from $y$ converge to $\delta_0$
weakly.
\end{example}

\begin{example}\label{ex52} This example shows that the assumptions of Theorem \ref{unique} or Corollary \ref{corounique} do not guarantee weak convergence of transition
probabilities. Take $E=\{0,1\}$ with transition probabilities $p_{0,1}=p_{1,0}=1$. The assumptions hold with $M=\{(0,0)\}$ respectively $M=\{0\}$ and $\alpha=1$ and there
exists a unique invariant measure $\mu$ but the transition probabilities do not converge to $\mu$. Note however that under the assumptions of Theorem \ref{unique} or
Corollary \ref{corounique}, for any ergodic invariant measure $\mu$ the (time-){\em averaged} transition probabilities converge to $\mu$ for $\mu$-almost all initial
conditions
$y \in E$ by the ergodic theorem.
\end{example}

\begin{example}\label{ex53} This example shows that Theorem \ref{conv2} fails if \eqref{prob1} is replaced by a corresponding averaged limit.
 Consider a deterministic dynamics on an unbounded countable subset $E=\{0,a_1,a_2,...\}$ of $[0,\infty)$ which maps 0 to 0 and $a_1 \to a_2 \to ...$. Choosing the
 sequence such that it has only 0 as an accumulation point we can ensure that the chain is Feller. Clearly, $\delta_0$ is an invariant measure. On the other hand,
 if  the average of the first $n$ members of the sequence converges to 0 then for every $x=a_j\in E$ and $y=0$ the (deterministic) coupling
 $X_n=a_{j+n}, Y_n=0, n\geq 0$ satisfies the averaged analogue of \eqref{prob1}. However, if $a_n\not \to 0$, we have $P_n(x,\cdot) \not\Rightarrow \delta_0$
 for each $x\not=0$.
\end{example}

\begin{example}\label{ex54} This example shows that an ergodic Feller chain is not necessarily an e-chain.
 Consider a deterministic dynamics on the unbounded countable set $E=\{0\}\cup\{2^{-k}, k\geq 0\}$ which maps 0 to 0, 1 to 0, and
$2^{-k}$ to $2^{-k+1}$, $k\geq 1$. Clearly,  the chain is Feller and for every $x\in E$, $P_n(x,\cdot)$ converge as $n\to \infty$ to the unique
invariant measure $\mu=\delta_0$ -- even in the total
 variation distance. However, for any two points $x,y\in E\setminus\{0,1\}$ with, say, $x>y$ there exists $n\geq 1$ such that
 $2^nx=1, 2^ny\in (0,1/2],$ and therefore
 $$
 \sup_n d(P_n(x,\cdot), P_n(y,\cdot))\geq {1\over 2}.
 $$
 On the other hand, for any $\delta>0$ there exist $k,m$ large enough so that $d(2^{-k}, 2^{-m})<\delta$. That is, this chain is not an e-chain.
Note that this example is also not {\em asymptotically strong Feller} (see \cite{HM06} for the definition of this concept). The example
does however satisfy the assumptions of all theorems in Section \ref{mainresults}.
\end{example}

\begin{example}\label{ex55} This example shows that Theorem \ref{conv2} may fail if the assumption $\pi_1(\xi_{x,y})\sim \P_x$ is omitted, and
\eqref{prob1} holds true just for $\xi_{x,y}\in \widehat C(\P_x, \P_y)$. Consider $E=\{0,1,2,...\}$ with transition probabilities
$p_{0,0}=1$ and $p_{i,i-1}=1/3$ and $p_{i,i+1}=2/3$ for $i=1,2,...$. Clearly, $\mu=\delta_0$ is the unique invariant measure, transition probabilities
$P_n(x,.)$ do not converge
to $\mu$ for $x \neq 0$ and for each $x \in \N$ there exists some $\xi\in \widehat  C(\P_x,\P_0)$ such that $X_n \to 0$ almost surely under $\xi$.
\end{example}

\begin{example}\label{comparison} This final example clarifies the relation between condition \eqref{prob1} and the condition
\begin{equation}\label{conv_as}
 \xi_{x,y}(\lim_{n\to \infty}d(X_n, Y_n)=0)>0,
\end{equation}
which was  used in \cite[Theorem 3.1]{HMS11}. Namely, we show that the ``convergence in probability'' type assumption \eqref{prob1} is strictly weaker
than the ``convergence with positive probability'' one \eqref{conv_as}. Since this difference may not be too crucial, in order not
to overburden the exposition we just outline the construction and omit detailed proofs.

Let $E=[0,1) \times \{-1,1\}$ be equipped with the  metric $d((u,i),(v,j))=\widetilde d(u,v)  +|j-i|$, where
$\widetilde d$ denotes the Euclidean metric on the torus $T=[0,1)$ and let $r \in (0,1)\backslash \Q$.
Define a Markov operator $P$ on $E$ as follows: for any $x=(u, i)\in E$,
\begin{eqnarray*}
P\big((u,i),\big\{ (u+r \mbox{ mod } 1,i)\big\} \big)&=&1/2\\
P\big( (u,i),\{(u,-i)\} \big)&=&1/2.
\end{eqnarray*}
It is clear that $P$ is Feller, and there exists at least one invariant probability measure, namely
$\mu=\lambda \otimes \big( \frac 12 \delta_0 + \frac 12 \delta_1\big)$, where $\lambda$ denotes Lebesgue measure on $T$.

In the following we distinguish between {\em components} and {\em coordinates}, the former referring to the first
or second element of a pair $(X,Y)\in E^\infty\times E^\infty$,  and the latter referring to the first or second element of a point $x=(u,i) \in E$. For any generalized
coupling $\xi_{x,y}\in \widehat{C}(\P_x, \P_y)$ the distance of the two components remains constant as long as
their second coordinates are the same, and if the second coordinates differ the distance is at least two. Therefore
the only way the distance of the two components can converge to zero is that they coincide eventually.
Hence for any $x=(u,i), y=(v,j)$ such that $u-v$ is {\em not} an integer multiple of $r$ (mod 1),  it is
clear that there is no $\xi \in \widehat  C(\P_x,\P_y)$ for which the distance between the two components converges to zero
with positive probability. That is, there are no sets $M_1,M_2 \in \EE$ of positive $\mu$-measure such that for each $(x,y) \in M_1 \times M_2$ there exists some
$\xi_{x,y} \in \widehat  C(\P_x,\P_y)$ satisfying \eqref{conv_as}.

On the other hand, for any $x,y\in E$ we can find a coupling $\xi_{x,y} \in  C(\P_x,\P_y)$ such that \eqref{prob1} holds
(and hence $P_n(x,\cdot)\Rightarrow\mu$ for any $x$). Fix $x,y\in E$ and define $\xi_{x,y} \in  C(\P_x,\P_y)$ as a Markov chain $\{(X_n, Y_n), n\geq 0\}$
defined as follows with the function $p(z),z \in [0,1)$ yet to be determined:
\begin{itemize}
  \item if the second coordinates of $X_{n}, Y_{n}$ differ,  then with probability $1/2$,   $X_{n+1}$ changes the second coordinate and $Y_{n+1}$  doesn't and the same
  holds for $X$ and $Y$ interchanged, so in both cases the second coordinates of $X_{n+1}, Y_{n+1}$ coincide;
  \item if the second coordinates of $X_{n}, Y_{n}$ coincide, and the difference (mod 1) between the first coordinates  of $X_n, Y_n$ is $z \in [0,1)$, then
   the second coordinates of  $X_{n+1}$ and  $Y_{n+1}$  either change or stay the same simultaneously, with the probability of each of these two possibilities $1/2(1-p(z))$, and   the probabilities that the second coordinate of  $X_{n+1}$ (resp. $Y_{n+1}$) changes while $Y_{n+1}$ (resp. $Y_{n+1}$) doesn't,  are equal $1/2p(z)$.
\end{itemize}
By construction, if at some moment the second coordinates differ, they become equal immediately afterwards. Consider the sequence $\{Z_n\}$  of differences of the first
coordinates of $\{(X_n, Y_n)\}$. If $Z_n=z$ and the second coordinates of $X_n$ and $Y_n$ coincide, then $Z$ will keep taking the value $z$ for a geometric number of steps
with expected value $1/p(z)$, then the second coordinates of $X,Y$ will be different for one time unit after which they become the same again and $Z$ takes the values
$z,\,z+2r$, and $z-2r$ with probabilities $1/2,\, 1/4$, and $1/4$ respectively. It is not hard to see (and easy to believe) that if the continuous function  $z \mapsto p(z)$ is chosen such that
$p(0)=0$, $p(z)>0$ for $z \neq 0$ and $p(z)$ approaches 0 as $z \to 0$ sufficiently fast, then both $Z_n$ and the indicator $1_{X_n\neq Y_n}$ will
converge to 0 in probability since $Z_n$ is very likely to take a value close to 0 when $n$ is large.
\end{example}

\section{Applications: SFDEs and SPDEs}\label{sSDDE}

In this section we illustrate our main results applying them to stochastic functional differential equations (SFDEs) and stochastic partial differential equations (SPDEs).

\subsection{Stochastic delay equations}\label{s31}
Denote $C:=C([-1,0],\R^m)$, and for a function or
a process $X$  defined on $[-1,t]$  write $X_t (s) := X(t + s), s \in [-1, 0]$. Consider the SFDE
\begin{align}\label{SFDE}
\dd X(t)&=F(X_t)\,\dd t + G(X_t)\,\dd W(t),\\
X_0&=f \in C,
\end{align}
where $F:C \to \R^m$ and $G:C \to \R^{m\times m}$ satisfy a global Lipschitz condition with respect to the supremum norm and $W$ is a
standard Wiener process in $\R^m$. Assume the  non-degeneracy condition
\begin{equation}\label{non-deg}
\sup_{f \in C}\big | G^{-1}(f))\big |<\infty,
\end{equation}
where $G^{-1}(f)$ denotes the generalized (or Moore-Penrose) inverse matrix of $G(f)$, $f\in C$.

This model was well studied in \cite{HMS11}, where it was proved that the  $C$-valued solution process $X_t$, $t \ge 0$ is uniquely defined, is a Feller process, and has at most
one invariant probability measure $\mu$ in which case all transition probabilities converge to $\mu$ weakly (for ease of exposition we
have imposed slightly stronger  assumptions on $F$ and $G$ compared to \cite{HMS11}).

Here we use this model to benchmark our results. Namely, we will show that these results can be applied yielding the same conclusions, but in a considerably easier and more straightforward way.

Like in  \cite{HMS11}, we fix a pair of initial conditions $f$ and $g$ in $C$, and consider the
pair of equations
$$
\begin{aligned}
\dd X(t)&=F(X_t)\,\dd t +G(X_t)\,\dd W(t),&X_0=f,\\
\dd Y(t)&=F(Y_t)\,\dd t + \lambda(X(t)-Y(t))\,\dd t + G(Y_t)\,\dd W(t),\quad &Y_0=g.
\end{aligned}
$$
It is shown in \cite{HMS11}, Section 3 that if $\lambda >0$ is sufficiently large (when compared with the Lipschitz constants for $f,g$),  then with probability 1
$$
|X(t)-Y(t)|\to 0, \quad t\to \infty
$$
exponentially fast, and thus
\begin{equation}\label{delta_2}
  \int_0^\infty \big |X(t)-Y(t)\big |^2 \,\dd t <\infty
\end{equation}
(the proofs are not very long and are based on basic stochastic calculus arguments). Observe that the equation for $Y$ can be re-written to the form
\begin{equation}\label{SDDE}
\dd Y(t)=F(Y_t)\,\dd t + \lambda(X(t)-Y(t))\,\dd t + G(Y_t)\,\dd \widehat W(t),\quad Y_0=g
\end{equation}
with
$$
\widehat W(t)=W(t)+\int_0^t\beta_s\,\dd s, \quad \beta_t:=\lambda(X(t)-Y(t))G^{-1}(Y_t).
$$
Combining \eqref{delta_2} with \eqref{non-deg}, we see that
$$
\int_0^\infty\beta_t^2\, dt<\infty
$$
with probability 1. Then by the Girsanov theorem the law of $\widehat W$ on $C([0, \infty), \R^m)$ is absolutely continuous w.r.t. the law of the Wiener process $W$; cf. \cite{LipSher},
Theorem 7.4. Because $Y$ is the strong solution to \eqref{SDDE}, this yields immediately that the law of
of $Y(t), t\in [-1, \infty)$ is absolutely continuous with
respect to the law of the solution to \eqref{SFDE} with initial condition $g$. On the other hand, $X$ is just the solution to \eqref{SFDE} with initial
condition $f$, hence the joint law $\xi$ of the pair $X,Y$ is a generalized coupling from the class $\widehat{C}(\P_f, \P_g)$ which satisfies
the additional condition $\pi_1(\xi)=\P_f$. Applying the continuous-time version of Corollary \ref{coroconv}, we directly obtain weak convergence of all transition probabilities
to the unique invariant probability measure
(in the case it exists).

We note that the simple construction explained  above can not be applied directly within the approach developed in \cite{HMS11}.
Theorem 3.1 in \cite{HMS11}, which provides uniqueness of the invariance measure, exploits a generalized coupling which belongs to the class $\widetilde{C}(\P_f, \P_g)$.
It is difficult to guarantee the \emph{equivalence} of the law of $Y$ to $\P_g$ using just the Girsanov theorem; this is the reason why in the proof of
uniqueness in \cite{HMS11}  a more sophisticated construction of the generalized coupling is used which involves localization in time. The proof of
Theorem 3.7 in \cite{HMS11}, which states the  weak convergence of transition probabilities to the invariant measure, contains an extra analysis which actually
shows that $X$ is an e-process. None of these additional considerations are required in our approach. This is a clearly seen advantage, which makes it possible to extend the uniqueness
results to asymptotic stability (almost) for free. Below we show that such a possibility is quite generic and is available  as well  in SPDE setting.

\subsection{SPDEs}\label{s32}

 In this section we show that in each of the five SPDE models studied in \cite{GMR15} only a minor modification of the constriction of a generalized coupling allows us to apply
 Corollary \ref{coroconv} and thus to obtain the asymptotic stability of the model rather than just unique ergodicity. Such a drastic improvement becomes possible thanks to
 Theorem \ref{conv1} and Theorem \ref{conv2}, and illustrates the usefulness of these results. To simplify the cross-references, within this section we mainly adopt the notation
 from \cite{GMR15} even if it does not correspond to the notation introduced in Section \ref{basic}. The methodology will be similar for all the five models, hence we explain most
 details for the first one and then just sketch the argument for the other four. Throughout this section we denote  by $H^r$ the Sobolev classes $H^r_2(\D)$ with a domain $\D$
 which varies from
 model to model. The $L_2$-norm and the $H^1$-norm are denoted $|\cdot|$ and $\|\cdot\|$ respectively, for all  other norms  are indicated explicitly.
 We also denote by $\lambda_n,  n\geq 1$ the increasingly enumerated eigenvalues of an operator $A$, which will be specified in each model separately, and by
 $P_N$ the  projector onto the span of the
 respective first $N$ eigenvectors.

\subsubsection{2D Navier-Stokes on a domain}

Consider the 2D stochastic Navier-Stokes equation on $\D\subset \R^2$
\begin{equation}\label{2dNS1}
\dd \mathbf{u} + \mathbf{u}  \cdot \nabla \mathbf{u} \,\dd t = (\nu\Delta \mathbf{u}  +\nabla \pi + \mathbf{f}) \,\dd t +\sum_{k=1}^m
\sigma_k \dd W_k,\quad  \nabla \cdot \mathbf{u}  = 0,
\end{equation}
with the unknown velocity field $\mathbf{u} = (u_1, u_2)$ and the unknown pressure $\pi$. The bounded domain  $\D$ is assumed to have smooth $\partial\D$, and  the
no-slip (Dirichlet)
boundary condition on $\mathbf{u}$ is imposed:
\begin{equation}\label{2dNS2}
\mathbf{u}|_{\partial\D} = 0.
\end{equation}
The deterministic vector fields $\mathbf{f}, \sigma_1, \dots, \sigma_m\in L_2(\D)^2$  and independent standard Brownian motions $W_1, . . . ,W_m$ are fixed.

Denote by $V$ the subspace of $H^1(\D)^2$,  which contains $\mathbf{u}$ such that $\nabla \cdot \mathbf{u}  = 0$ and $\mathbf{u}\cdot \mathbf{n}  = 0$
(where $\mathbf{n}$ denotes the
outward normal for $\partial\D$). Denote by $H$ the completion of $V$ w.r.t. the $L_2(\D)^2$-norm, by $P_H$ the projector in $L_2(\D)^2$ on $H$, and by $A=-P_H\Delta$
the \emph{Stokes operator}.

It is known that for any $\mathbf{u}_0\in H$ the system \eqref{2dNS1}, \eqref{2dNS2} with the initial data $\mathbf{u}_0\in H$
admits a unique (strong) solution with values in $H$, and this solution  depends continuously on $\mathbf{u}_0\in H$. That is, \eqref{2dNS1}, \eqref{2dNS2} defines a
Feller Markov process
valued in $H$; we refer for details to \cite{GMR15}, Section 3.1.1.

Now we explain the  generalized coupling construction for this system. Fix arbitrary $\mathbf{u}_0,\widetilde{\mathbf{u}}_0\in H$ and consider
$\mathbf{u}=\mathbf{u}(\cdot , \mathbf{u}_0)$
solving \eqref{2dNS1}, \eqref{2dNS2} with initial data $\mathbf{u}_0$, and $\widetilde{\mathbf{u}}$ solving
$$
\dd\widetilde{\mathbf{u}} +\widetilde{\mathbf{u}}  \cdot \nabla\widetilde{\mathbf{u}} \,\dd t = (\nu\Delta\widetilde{\mathbf{u}} +
\lambda P_N(\mathbf{u} - ˜\widetilde{\mathbf{u}}) +
\nabla \varpi + \mathbf{f}) \,\dd t +\sum_{k=1}^m
\sigma_k \dd W_k,\quad  \nabla \cdot\widetilde{\mathbf{u}}  = 0,
$$
$$
\widetilde{\mathbf{u}}|_{\partial\D} = 0
$$
with initial data $\widetilde{\mathbf{u}}_0$; here $\lambda$ and the number $N$
 are yet to be chosen; recall that $P_N$ is the projector which is defined in the terms of the operator $A$.

 This   construction is based on the ``stochastic control'' argument, similar to the one developed in \cite{HMS11}, Section 3; see also
 \cite{GMR15}, Section 2.4.
 The similar coupling construction in \cite{GMR15}, Section 3.1.2 involves an additional localization term $1_{\tau_K>t}$, and the corresponding
 generalized coupling is
 defined as the conditional law of the pair $({\mathbf{u}}, \widetilde{\mathbf{u}})$ on the set $\{\tau_K=\infty\}$. This gives a generalized
 coupling from the class $\widehat C(\P_{\mathbf{u}}, \P_{\widetilde{\mathbf{u}}_0})$. Because of the conditioning, the law of the first component have
 no reason to be equivalent to $\P_{\mathbf{u}}$. The latter condition is however crucial for our Theorem \ref{conv2}; see Remark \ref{rem29} and Example \ref{ex55}.
 We resolve this difficulty in a similar way we did in Section \ref{s31}. Namely, we remove the localization term and consider the law of the pair
 $({\mathbf{u}}, \widetilde{\mathbf{u}})$ as the required generalized coupling. This leads only to minor modifications in the respective calculus, as we
 explain below, but it allows to apply our main results in order to derive asymptotic stability.

  The difference $\mathbf{v}:=\mathbf{u}-\widetilde{\mathbf{u}}$  satisfies
\begin{equation}\label{eq}
\dd \mathbf{v}-\nu \Delta  \mathbf{v}\, \dd t+1_{\tau>t}\lambda P_N\mathbf{v}\, \dd t=-\nabla \pi+\nabla\varpi
+\widetilde{\mathbf{u}}\cdot \nabla \mathbf{u} +\mathbf{u}\cdot \nabla\widetilde{\mathbf{u}}, \quad \nabla \cdot \mathbf{v}=0, \quad \mathbf{v}|_{\partial\D} = 0.
\end{equation}
Like in \cite{GMR15}, Section 3.1.2,  multiplying \eqref{eq} by $\mathbf{v}$, integrating over $\D$, and using that $\mathbf{u},\widetilde{\mathbf{u}},$ and
$\mathbf{v}$ are all divergence free and satisfy the Dirichlet boundary condition, one gets $$
\begin{aligned}
\frac{1}{2}\dd |\mathbf{v}|^2 + \nu\|\mathbf{v}\|^2\, \dd t+  \lambda |P_N \mathbf{v}|^2 \dd t
&\leq \left|\int_{\D} \mathbf{v}\cdot \nabla \mathbf{u}\cdot \mathbf{v}\, \dd x\right|\, \dd t
\\&\leq C_\D |\mathbf{v}|\|\mathbf{v}\|\|\mathbf{u}\|\, \dd t\leq \left(\frac{\nu}{2}\|\mathbf{v}\|^2+
\frac{C_\D}{2\nu}|\mathbf{v}|^2\|\mathbf{u}\|^2\right)\, \dd t
\end{aligned}
$$
with a universal constant $C_\D$ which  involves the quantities  from Sobolev embedding. By the Poincar\'e inequalities \cite{GMR15} (3.3),
for the particular choice $\lambda =\nu \lambda_N/2$ we get
$$
\lambda |P_N \mathbf{v}|^2+\frac{\nu}{2}\|\mathbf{v}\|^2\geq \lambda|\mathbf{v}|^2, \quad \lambda\leq \nu \lambda_N/2.
$$
Taking $\lambda=\nu \lambda_N/2$ we obtain
$$
\dd |\mathbf{v}|^2\leq \Big(-\nu \lambda_N1_{\tau>t}+\frac{C_\D}{\nu}\|\mathbf{u}\|^2\Big)|\mathbf{v}|^2\, \dd t,
$$
and finally by Gronwall's lemma
\begin{equation}\label{2Dexp}
|\mathbf{v}(t)|^2\leq |\mathbf{u}_0-\widetilde{\mathbf{u}}_0|^2\exp\left(-\nu\lambda_N t+\frac{C_\D}{\nu}\int_0^t\|\mathbf{u}(s)\|^2\, \dd s\right), \quad t\geq 0.
\end{equation}

One has  with probability 1
\begin{equation}\label{liminf}
  \limsup_{t\to \infty}\frac{1}{t}\int_0^t \|\mathbf{u}(s)\|^2\, \dd s\leq  \frac{|A^{-1/2}\mathbf{f}|^2}{2\nu^2}+\frac{|\sigma|^2}{\nu},
\end{equation}
where  $|\sigma|^2:= \sum_{k=1}^m |\sigma_k|^2$; this follows from the energy estimate \cite{GMR15} (3.5). Hence $N$ satisfies
\begin{equation}\label{sup}
\lambda_N>C_\D\left(\frac{|A^{-1/2}\mathbf{f}|^2}{2\nu^4}+\frac{|\sigma|^2}{\nu^3}\right),
\end{equation}
with probability 1 the right hand side term in \eqref{2Dexp} tends to 0 exponentially fast.

On the other hand, consider
$\sigma$ as a linear operator $\R^m\to H$ and assume that, for the given $N$,
\begin{equation}\label{range}
H_N:=P_NH\subset\mathrm{Range}\, (\sigma)=\mathrm{Span}\, (\sigma_k, k=1, \dots, m).
\end{equation}
Then the corresponding pseudo-inverse operator $\sigma^{-1}: H_N\to \R^m$ is well defined and bounded. Then the principal equation for $\widetilde{\mathbf{u}}$ can be written
in the form
$$
\dd\widetilde{\mathbf{u}} +\widetilde{\mathbf{u}}  \cdot \nabla\widetilde{\mathbf{u}} \,\dd t = (\nu\Delta\widetilde{\mathbf{u}} + \nabla \varpi + \mathbf{f}) \,\dd t +\sum_{k=1}^m
\sigma_k \dd \widetilde W_k
$$
with
$$
\widetilde W(t)=W(t)+\int_0^t\beta_s\,\dd s, \quad \beta_t:=\lambda \sigma^{-1}P_N\mathbf{v}(t)
$$

Since $\sigma^{-1}$ is bounded and $|v(t)|$ tends to $0$ exponentially fast, we have
\begin{equation}\label{beta}
\P\left(\int_0^t\|\beta_s\|^2_{\R^m}\,\dd s<\infty\right)=1,
\end{equation}
 and the law of $\widehat W$ is absolutely continuous w.r.t. the law of $W$. Thus the law of $\widetilde{\mathbf{u}}$  is absolutely continuous w.r.t. the law of the solution
 to \eqref{2dNS1}, \eqref{2dNS2} with initial data $\widetilde{\mathbf{u}}_0$. Note that the law of the first component w.r.t. this coupling just equals $\P_{\mathbf{u}_0}$
 and the distance between the components tend to $0$ exponentially fast as $t\to \infty$. Hence  the law of the pair $(\mathbf{u}(\cdot),\widetilde{\mathbf{u}}(\cdot))$ can
 be used as the coupling required in Corollary \ref{coroconv}  with $E=H,\rho=d=|\cdot-\cdot|\wedge 1$. We conclude that  in the framework of  Proposition 3.1 \cite{GMR15}, which
 states unique ergodicity for \eqref{2dNS1}, \eqref{2dNS2},  the following stabilization property actually holds true:
\begin{center}
\emph{for any $\mathbf{u}\in H$, the transition probabilities $P_t(\mathbf{u}, \cdot)\in \mathcal{P}(H)$
weakly converge as $t\to \infty$ to the unique invariant measure.}
\end{center}

\subsubsection{2D Hydrostatic Navier-Stokes Equations}

Next, following \cite{GMR15} Section 3.2, we  consider a stochastic version of the 2D Hydrostatic Navier-Stokes equation
\begin{equation}\label{HNS1}\begin{aligned}
  \dd u&+ (u\partial_x u+w\partial_zu+\partial_xp-\nu\Delta u)\, \dd t =\sum_{k=1}^m \sigma_k\dd W_k,\\
\prt_zp &= 0,\\
\prt_xu &+\prt_z w = 0
\end{aligned}
\end{equation}
for an unknown velocity field $(u,w)$ and pressure $p$ evolving on the domain $\D = (0,L)\times(−h, 0)$. The boundary $\prt D$ is decomposed into its vertical sides
$\Gamma_v = [0,L] \times \{0,−h\}$ and lateral sides $\Gamma_l = \{0,L\} \times [−h, 0]$, where
the boundary conditions  are imposed:
\begin{equation}\label{HNSboundary}
  u = 0 \quad \hbox{on $\Gamma_l$}, \quad  \prt_zu = w = 0 \quad \hbox{on $\Gamma_v$}.
\end{equation}
Denote
$$
H=\left\{f\in L_2(\D):\int_{-h}^0u\, \dd z\equiv 0\right\}, \quad V=\left\{u\in H^1(\D):\int_{-h}^0u\, \dd z\equiv 0, \, u|_{\Gamma_l}=0\right\}.
$$
Denote also by $P_H$ the projector in $L_2(\D)$ on $H$, and put $A=-P_H\Delta$.

 It is known (see \cite{GMR15}, Section 3.2.1) that under a proper condition on the family $\{\sigma_k\}$ for a given $u_0\in V$ the system (\ref{HNS1}), \eqref{HNSboundary}
 has a unique strong solution, which in addition depends continuously on $u_0\in V$. Thus the system (\ref{HNS1}), \eqref{HNSboundary} defines a Feller Markov process in $E=V$.

Now we explain the generalized coupling construction for (\ref{HNS1}), \eqref{HNSboundary}.
For fixed $u_0, \widetilde u_0\in V$, consider the solution $u$ to (\ref{HNS1}), \eqref{HNSboundary} with the initial data $u_0$ and the solution $\widetilde u$  to a similar
system with the first equation changed to
$$
\dd \widetilde u+ (\widetilde u\partial_x \widetilde u+w\partial_z \widetilde u+\partial_x \widetilde p-\nu\Delta \widetilde u+\lambda P_N (u-\widetilde u))\, \dd t
=\sum_{k=1}^m \sigma_k\dd W_k
$$
with $\lambda=\nu\lambda_N/2$. One has
\begin{equation}\label{Hlip}
|v(t)|\leq \exp\left(-2\lambda t+C\int_0^t\Big(\|u(s)\|^2+|\prt_z u(s)|\|\prt_z u(s)\|\Big)\, \dd s\right)|v(0)|, \quad t\geq 0
\end{equation}
with a constant $C$ depending only on $\nu$ and $\D$; see (3.22), \cite{GMR15}. Next, there exists $C_1$ depending only on $\nu, \D$, and $|\sigma|^2+|\prt_z\sigma|^2$ such that
\begin{equation}\label{Hliminf}
  \limsup_{t\to \infty}\frac{1}{t}\int_0^t \Big(\|u(s)\|^2+\|\prt_z u(s)\|^2\Big)\, \dd s\leq  C_1
\end{equation}
with probability 1; this follows from the energy estimates (3.16), (3.18) \cite{GMR15}. If  $N$ is large enough, so that $2\lambda=\nu\lambda_N>C C_1$, the above inequalities
yield that the $H$-norm $|v(t)|$ tends to zero as $t\to \infty$ exponentially fast. If in addition for such $N$ \eqref{range} holds true, then one can interpret $\widetilde u$
as the solution to (\ref{HNS1}), \eqref{HNSboundary} with the initial data $\widetilde u_0$ and $W$ changed to
$$
\widetilde W(t)=W(t)+\int_0^t\beta_s\,\dd s, \quad \beta_t:=\lambda \sigma^{-1}P_N {v}(t).
$$
Since the pseudo-inverse operator $\sigma^{-1}:H_N\to \R^m$ is bounded and $|v(t)|$ decays exponentially fast, we have (\ref{beta}). Hence the law of $\widetilde W$ is
absolutely continuous w.r.t. the law of $W$ and therefore  the law of $\widetilde u$ in $C([0, \infty), V)$ is  absolutely continuous w.r.t. $\P_{\widetilde u_0}$. Recall
that  the Markov process which corresponds to (\ref{HNS1}), \eqref{HNSboundary} is  well defined and is Feller on $V$. However, it is an easy observation that this process is
$H$-Feller, as well. Namely, inequality \eqref{Hlip} actually holds true for any $\lambda\leq \nu\lambda_N/2$, and taking $\lambda=0$ we easily deduce the $H$-continuity
of the semigroup.

We  take $E=V$,  $\rho=\|\cdot-\cdot\|\wedge 1, d=|\cdot-\cdot|\wedge 1$; note that  condition \eqref{approx} holds true with  $\rho_n^y(x)=|P_n(y-x)|, n\geq 1, y\in E.$   In this setting, we apply continuous time version of Corollary \ref{coroconv} with the generalized coupling $\xi$
 defined as the joint law of processes $u(\cdot), \widetilde u(\cdot)$ defined above. We conclude that in the framework of  Proposition 3.2 \cite{GMR15}, which states unique
 ergodicity for (\ref{HNS1}), \eqref{HNSboundary}, in addition the following (weak) $L_2$-stabilization property holds:
\begin{center}
\emph{for any $u\in V$, the transition probabilities $P_t(u, \cdot)\in \mathcal{P}(V)$
weakly converge in the $L_2$-topology as $t\to \infty$ to the unique invariant measure $\mu\in\mathcal{P}(V)$.}
\end{center}

\subsubsection{The fractionally dissipative Euler model}

Next, following \cite{GMR15} Section 3.3, we consider  the fractionally dissipative Euler model, described by the system
\begin{equation}\label{dEu}
\dd \xi+\left(\Lambda^\gamma\xi+\mathbf{u}\cdot \nabla\xi\right)\, \dd t=\sum_{k=1}^m\sigma_k \dd W_k, \quad \mathbf{u}=\K*\xi
\end{equation}
for an unknown vorticity field $\xi$ (this is the notation borrowed from \cite{GMR15}, which is not to be mixed with the notation for a coupling we used previously).
Here $\Lambda^\gamma = (-\Delta)^\gamma$ is the fractional Laplacian with $\gamma\in (0, 2]$, $\K$  is the Biot-Savart kernel, so that $\nabla^\perp \cdot \mathbf{u}= \xi$
and $\nabla \mathbf{u} = 0$, and \eqref{dEu} is posed on the
periodic box $\TT^2 = [−\pi, \pi]^2$. In the velocity formulation, \eqref{dEu} has the form
\begin{equation}\label{dEu2}
\dd \mathbf{u} + \Big(\Lambda^\gamma \mathbf{u} + \mathbf{u}\cdot \nabla \mathbf{u}+ \nabla \pi\Big) \,\dd t = \sum_{k=1}^m
\sigma_k \dd W_k,\quad  \nabla \cdot \mathbf{u}  = 0,
\end{equation}
where the unknowns are the velocity field $\mathbf{u}$ and the pressure $\pi$.
It is known that for a fixed $r>2$ for any given $\mathbf{u}_0\in H^r$ there exists a unique strong solution to \eqref{dEu2} taking values in $H^r$ and this solution
depends on the initial data $\mathbf{u}_0\in H^r$ continuously. That is,   \eqref{dEu2} defines a Feller Markov process valued in $H^r$; see \cite{GMR15}, Section 3.3.1.

For fixed $\mathbf{u}_0, \widetilde{\mathbf{u}}_0\in H^r$, consider the function $\mathbf{u}(\cdot)$ solution to \eqref{dEu2} with the initial data $\mathbf{u}_0$ and the
function $\widetilde{\mathbf{u}}(\cdot)$ solving
$$
\dd\widetilde{\mathbf{u}} + \Big(\Lambda^\gamma\widetilde{\mathbf{u}}-\lambda P_N(\mathbf{u}-\widetilde{\mathbf{u}}) +\widetilde{\mathbf{u}}\cdot \nabla \widetilde{\mathbf{u}}+
\nabla \widetilde{\pi}\Big) \,\dd t = \sum_{k=1}^m
\sigma_k \dd W_k,\quad  \nabla \cdot \widetilde{\mathbf{u}}  = 0
$$
with the initial data $\widetilde{\mathbf{u}}_0$ and $P_N$ which now denotes the projector which corresponds to the eigenfunctions of $A=\Lambda^\gamma$. This is actually the
generalized coupling construction from \cite{GMR15}, Section 3.3.2, where in the additional control term we remove the localization term $1_{\tau_K>t}$. Denote
$\mathbf{v}=\mathbf{u}-\widetilde{\mathbf{u}} $, then for $\lambda\leq \lambda_N/2$
$$
|\mathbf{v}(t)|^2\leq \exp\left(-2\lambda t+C\int_0^t\|\xi(s)\|^2_{L_p}\, \dd s\right)|\mathbf{v}(0)|^2, \quad t\geq 0
$$
with properly chosen $p>1$ and universal $C$;  see  \cite{GMR15}, (3.28). On the other hand, there exists a universal $C_1$ such that
\begin{equation}\label{dEuliminf}
  \limsup_{t\to \infty}\frac{1}{t}\int_0^t \|\xi(s)\|^2_{L_p}\, \dd s\leq  C_1\|\sigma\|^2_{L_p}
\end{equation}
with probability 1, where
$$
\|\sigma\|_{L_p}=\left(\int_{\TT^2}\left(\sum_{k=1}^m\sigma_k^2\right)^{p/2}\, \dd x\right)^{1/p}.
$$
This follows from  the energy estimate (3.31)  \cite{GMR15}.

Now we can repeat literally the argument from the previous subsection. Taking in the above calculation $\lambda=0$, we see that the Markov process is $H$-Feller
with $H=L_2(\TT^2)$. Taking $\lambda=\lambda_N/2$, we get that if, for some $N$, $\lambda_N>C C_1$ and \eqref{range} holds, then the law of the pair
$(\mathbf{u}(\cdot),\widetilde{\mathbf{u}}(\cdot))$ can be used as the coupling required in Corollary \ref{coroconv} with
$E=H^r,\rho=\|\cdot-\cdot\|_{H^r}\wedge 1, d=|\cdot-\cdot|\wedge 1$ (again, condition \eqref{approx} is easy to verify). We conclude that  in the framework of  Proposition 3.3 \cite{GMR15}, which states unique ergodicity
for  \eqref{dEu2}, in addition the following (weak) $L_2$-stabilization property holds:
\begin{center}
\emph{for any $u\in H^r$, transition probabilities $P_t(u, \cdot)\in \mathcal{P}(H^r)$
weakly converge in the $L_2$-topology as $t\to \infty$ to the unique invariant measure $\mu\in\mathcal{P}(H^{r})$.}
\end{center}

\subsubsection{The damped stochastically forced Euler-Voigt model}

Next, we consider  an inviscid
``Voigt-type'' regularization of a damped stochastic Euler equation:
\begin{equation}\label{EV}
\dd \mathbf{u} + \Big(\gamma  \mathbf{u} + \mathbf{u}_\alpha\cdot \nabla \mathbf{u}_\alpha+ \nabla p\Big) \,\dd t = \sum_{k=1}^m
\sigma_k \dd W_k,\quad  \nabla \cdot \mathbf{u}  = 0,
\end{equation}
with for some $\gamma>0$ and the unknown vector field $\mathbf{u}$, where the
non-linear terms are subject to an $\alpha$-degree regularization
$$
(-\Delta)^{\alpha/2}\mathbf{u}_\alpha=\Lambda^\alpha\mathbf{u}_\alpha=\mathbf{u}.
$$
The absence of a parabolic regularization mechanism brings specific difficulties to the  analysis of the model, we refer to \cite{GMR15},
Sections 3.4.1 -- 3.4.3 for details. Surprisingly, the construction of the generalized coupling which leads to the stability of the model does not bring
substantial novelties and can be provided within the same lines we discussed previously. To shorten the exposition, we consider only the case of a 2D model
evolving on the periodic box $\TT^2$, and assume $\alpha>2/3$. In this case, Proposition 3.4, \cite{GMR15} shows that
for any $\mathbf{u}_0\in H^{1-\alpha/2}$ there exists unique strong solution to \eqref{EV} with the initial data, and the corresponding semigroup is
Feller w.r.t. $H^{-\alpha/2}$ norm.

Next, for fixed $\mathbf{u}_0, \widetilde{\mathbf{u}}_0\in H^{1-\alpha/2}$, consider the function $\mathbf{u}(\cdot)$ solution to \eqref{EV} with the initial data
$\mathbf{u}_0$ and the function $\widetilde{\mathbf{u}}(\cdot)$ solving
$$
\dd \widetilde{\mathbf{u}} + \Big(\gamma  \mathbf{u} -\lambda P_N(\mathbf{u}-\widetilde{\mathbf{u}}) + \widetilde{\mathbf{u}}_\alpha\cdot \nabla
\widetilde{\mathbf{u}}_\alpha+ \nabla \widetilde {p}\Big) \,\dd t = \sum_{k=1}^m
\sigma_k \dd W_k,\quad  \nabla \cdot \mathbf{u}  = 0
$$
with the initial data $\widetilde{\mathbf{u}}_0$; now $P_N$ denotes the projector on the span of $N$ first elements in the sinusoidal basis. Since $\alpha>2/3$,
there exists $\delta>0$ such that $H^{\alpha/2-\delta}\subset L^3$. For such $\delta$, inequalities (3.45), (3.46) \cite{GMR15} provide  the following bound for
$\mathbf{v}=\mathbf{u}-\widetilde{\mathbf{u}} $:
$$
{1\over 2}\dd \|\mathbf{v}\|^2_{H^{-\alpha/2}}+\Big(\gamma-C\bigg(\lambda^{-1}+N^{-\delta}\Big)\Big(1+\|\xi\|^2_{H^{-\alpha/2}}\Big)\bigg)\|\mathbf{v}\|^2_{H^{-\alpha/2}}\,
\dd t\leq 0,
$$
where $\xi=\mathrm{curl}\, \mathbf{v}$ and constant $C$ depends only on $\delta, \alpha$. On the other hand, with probability 1
$$
\limsup_{t\to \infty}\frac{\gamma}{t}\int_0^t \|\xi(s)\|^2_{H^{-\alpha/2}}\, ds\leq \|\varsigma\|^2_{H^{-\alpha/2}}
$$
with $\varsigma= \mathrm{curl}\, \sigma$; see \cite{GMR15}, Section 3.4.1. Hence, if $\lambda, N$ are taken large enough, $\|\mathbf{v}(t)\|^2_{H^{-\alpha/2}}$
tends to 0 as $t\to \infty$ exponentially fast. Repeating literally the same arguments as before we obtain  that, if $H_N\subset \mathrm{Range}(\sigma)$, the law of
the pair  $(\mathbf{u}(\cdot),\widetilde{\mathbf{u}}(\cdot))$ can be used as the coupling required in Corollary \ref{coroconv} with
$E=H^{1-\alpha/2}, \rho=\|\cdot-\cdot\|_{H^{1-\alpha/2}}\wedge 1, d=\|\cdot-\cdot\|_{H^{-\alpha/2}}\wedge 1$ (again, condition \eqref{approx} is easy to verify). We conclude that  in the framework of  Proposition 3.4 \cite{GMR15}, which
states unique ergodicity for  \eqref{EV}, in addition the following (weak) $H^{-\alpha/2}$-stabilization property holds:
\begin{center}
\emph{for any $u\in H^{1-\alpha/2}$, the transition probabilities $P_t(u, \cdot)\in \mathcal{P}(H^{1-\alpha/2})$
weakly converge in the $H^{-\alpha/2}$-topology as $t\to \infty$ to the unique invariant measure $\mu\in\mathcal{P}(H^{1-\alpha/2})$.}
\end{center}

\subsubsection{The damped nonlinear wave equation}

Finally, we consider the damped Sine-Gordon equation which is  written
as the system of stochastic partial differential equations
\begin{equation}\label{SG}
\dd v + \Big(\alpha v - \Delta u +\beta\sin(u)\Big)\, \dd t =  \sum_{k=1}^m
\sigma_k \dd W_k,\quad  \dd u=v\, \dd t,
\end{equation}
where the unknown $u$ evolves on a bounded domain $\D\subset \R^n$ with smooth boundary, and satisfies the Dirichlet boundary
condition $u|_{\prt\D}\equiv0$. The parameter $\alpha$  is strictly positive and $\beta$ is a  real number.

It is known (see \cite{GMR15}, Section 3.5.1) that for any initial data $U_0=(u_0, v_0)\in X:=H^1_0(\D)\times L_2(\D)$ there exists a unique strong solution to \eqref{SG},
and moreover \eqref{SG} defines a Feller Markov process in $X$. In this final example the generalized coupling construction proposed in \cite{GMR15}, Section 3.5.2 is
already well adapted for our purposes. Within this construction, they put
$$
\dd \widetilde{v} + \Big(\alpha \widetilde{v} - \Delta \widetilde{u} +\beta\sin(\widetilde{u})-\beta 1_{\tau_K>t}P_N(\sin({u})-\sin(\widetilde{u}))\Big)\, \dd t =  \sum_{k=1}^m
\sigma_k \dd W_k,\quad  \dd u=v\, \dd t,
$$
with the initial data $(\widetilde{u}_0,\widetilde{v}_0)$ and
$$
\tau_K=\inf\left\{t:\int_0^t|u(t)-\widetilde u(s)|^2\, ds\geq K\right\}.
$$
They prove that for $N,K$ sufficiently large $\tau_K=\infty$ a.s., and the difference $w=u-\widetilde u$ satisfies
$$
\|w(t)\|^2+|\prt_t w(t)|^2\to 0, \quad t\to \infty
$$
exponentially fast.  This means that, if $H_N\subset \mathrm{Range}(\sigma)$,  the joint law of the solutions $U=(u,v), \widetilde U=(\widetilde{u},\widetilde{v})$
can be used as the coupling required in Corollary \ref{coroconv} with $E=X$, $\rho=d=\|\cdot-\cdot\|_{X}\wedge 1$.
We conclude that  in the framework of  Proposition 3.5 \cite{GMR15}, which states unique ergodicity for  \eqref{SG}, in addition the following stabilization property holds:
\begin{center}
\emph{for any $(u,v)\in X=H^1_0(\D)\times L_2(\D)$, transition probabilities $P_t\big((u,v), \cdot\big)\in \mathcal{P}(X)$
weakly converge  as $t\to \infty$ to the unique invariant measure.}
\end{center}

\appendix

\section{Proofs of Propositions \ref{newlemma} and \ref{singular}}

\begin{proof}[Proof of Proposition \ref{newlemma}]
 \emph{I.} Take an arbitrary $\xi\in \widehat  C^R_p(\P,\Q)$, and consider the sets
$$
B_\gamma^1=\left\{x: {\dd \pi_1(\xi)\over \dd\P}(x)\leq \gamma^{-1}\right\},\quad  B_\gamma^2=\left\{x: {\dd \pi_2(\xi)\over \dd\Q}(x)\leq \gamma^{-1}\right\},\quad
C_\gamma=B_\gamma^1\times B_\gamma^2, \quad \gamma\in (0,1).
$$
Define the sub-probability measure $\eta_\gamma$ on $(E^\infty\times E^\infty, \mathcal{E}^{\otimes \infty}\otimes \mathcal{E}^{\otimes \infty})$ by
$$
\eta_\gamma(A)=\gamma \xi(A\cap C_\gamma).
$$
Then the ``marginal distributions''  $\pi_i(\eta_\gamma), i=1,2$ (which now are sub-probability measures, as well) satisfy
$$
\pi_1(\eta_\gamma)\leq \P, \quad \pi_2(\eta_\gamma)\leq \Q.
$$
Denote
$$
\beta_\gamma=\eta_\gamma(E^\infty\times E^\infty)=\gamma\xi(C_\gamma)\leq \gamma<1,
$$
then each of the measures  $\P-\pi_1(\eta_\gamma), \Q-\pi_2(\eta_\gamma)$ has total mass $1-\beta_\gamma.$ We put
\begin{equation}\label{split}
\zeta_\gamma=\eta_\gamma+(1-\beta_\gamma)^{-1}\big(\P-\pi_1(\eta_\gamma)\big)\otimes \big(\Q-\pi_2(\eta_\gamma)\big),
\end{equation}
which by construction belongs to $C(\P,\Q)$. Let us show that $\gamma$ can be chosen small enough, so that $\zeta=\zeta_\gamma$ possesses the required property.

Let $\alpha >0$. For $A \in \EE \otimes \EE$ satisfying $\xi(A)\ge \alpha$, we have
$$
\zeta_\gamma(A)\geq \eta_\gamma(A)\geq \gamma\Big(\alpha- \xi\big((E^\infty\times E^\infty)\setminus C_\gamma\big)\Big).
$$
Next,
$$
\xi\big((E^\infty\times E^\infty)\setminus C_\gamma\big)\leq  \xi\big((E^\infty\setminus B_\gamma^1)\times E^\infty\big)+
 \xi\big(E^\infty\times (E^\infty\setminus B_\gamma^2)\big)=\pi_1(\xi)\big(E^\infty\setminus B_\gamma^1\big)+\pi_2(\xi)\big(E^\infty\setminus B_\gamma^2\big)
$$
and by the definition of the sets $B_\gamma^i, i=1,2$
$$
\pi_1(\xi)\big(E^\infty\setminus B_\gamma^1\big)=\int_{E^\infty\setminus B_\gamma^1}  {\dd \pi_1(\xi)\over \dd\P}\,\dd\P\leq \gamma^{p-1}
\int_{E^\infty}\left({\dd \pi_1(\xi)\over \dd\P}\right)^p\,\dd\P\leq \gamma^{p-1}R^p,
$$
$$
\pi_2(\xi)\big(E^\infty\setminus B_\gamma^2\big)=\int_{E^\infty\setminus B_\gamma^2}  {\dd \pi_2(\xi)\over \dd\Q}\,\dd\Q\leq \gamma^{p-1}
\int_{E^\infty}\left({\dd \pi_2(\xi)\over \dd\Q}\right)^p\,\dd\Q\leq \gamma^{p-1}R^p.
$$
Hence, if $\gamma$ is taken small enough for  $4\gamma^{p-1}R\leq \alpha$, for every $A$ with $\xi(A)\geq \alpha$ we have for $\zeta=\zeta_\gamma$
$$
\zeta(A)\geq {\gamma\alpha\over 2}=:\alpha',
$$
which completes the proof of statement I.

\emph{II.} We fix $\xi\in \widehat  C(\P,\Q)$ and modify slightly the construction from the previous part of the proof. Let $B_\gamma^i, i=1,2, C_\gamma$ be as above,
then we define
$$
\widetilde\eta_\gamma(A)=\xi(A\cap C_\gamma).
$$
We fix  $\gamma\in (0,1)$ small enough, so that
$$
\xi\big((E^\infty\times E^\infty)\setminus C_\gamma\big)\leq \alpha/2,
$$
where $\alpha$ is as in the statement of the lemma.

We have $\widetilde\eta_\gamma=\gamma^{-1}\eta_\gamma$, and thus the total mass of the measure $\widetilde\eta_\gamma$ equals $\gamma^{-1}\beta_\gamma=\xi(C_\gamma)\leq 1$. In addition,
$$
\pi_1(\widetilde\eta_\gamma)\leq \gamma^{-1}\P, \quad \widetilde\pi_2(\eta_\gamma)\leq \gamma^{-1}\Q,
$$
and the total mass for each of the measures $\gamma^{-1}\P-\pi_1(\widetilde\eta_\gamma), \gamma^{-1}\Q-\pi_2(\tilde\eta_\gamma)$
equals $\gamma^{-1}(1-\beta_\gamma)\geq \gamma^{-1}-1$. Then
$$
\widetilde\zeta_\gamma:=\widetilde\eta_\gamma+(1-\gamma^{-1}\beta_\gamma)\Big(\gamma^{-1}(1-\beta_\gamma)\Big)^{-2}\Big(\gamma^{-1}\P-\pi_1(\widetilde \eta_\gamma)\Big)\otimes
\Big(\gamma^{-1}\Q-\pi_2(\widetilde\eta_\gamma)\Big)
$$
is a probability measure with
$$
\widetilde\zeta_\gamma(A)\geq \xi(A)-\xi\big((E^\infty\times E^\infty)\setminus C_\gamma\big)\geq \xi(A)-{\alpha\over 2};
$$
that is, $\widetilde\zeta_\gamma(A)\geq \alpha':=\alpha/2$ as soon as $\xi(A)\geq \alpha$. In addition,
the marginal distributions of  $\widetilde\zeta_\gamma$ equal
$$
(1-\beta_\gamma)^{-1}\Big((1-\gamma^{-1}\beta_\gamma)\P+(1-\gamma)\pi_1(\widetilde\eta_\gamma)\Big),\quad (1-\beta_\gamma)^{-1}\Big((1-\gamma^{-1}\beta_\gamma)\Q+
(1-\gamma)\pi_2(\widetilde\eta_\gamma)\Big),
$$
and their Radon-Nikodym densities w.r.t. $\P,\Q$ respectively  are bounded by
$$
R:=(1-\beta_\gamma)^{-1}\Big((1-\gamma^{-1}\beta_\gamma)+(1-\gamma)\gamma^{-1}\Big)=\gamma^{-1},
$$
hence $\widetilde \zeta_\gamma\in \widehat C_p^R(\P,\Q)$ for every $p\geq 1$.
\end{proof}

\begin{proof}[Proof of Proposition \ref{singular}]
 There exist two increasing sequences  $K_1^n,\, K_2^n, n\geq 1$ of  compact subsets of $E$ such that $K_1^n\cap K_2^n=\emptyset$,  $\nu_1(K_1^n)\ge 1-1/n$, and  $\nu_2(K_2^n)\ge 1-1/n$ for $n\geq 1$. Let
$$
\delta_n=d(K_1^n,K_2^n), \quad n\geq 1.
$$
Clearly $\delta_n, n\geq 1$ is non-increasing and $\delta_n>0$
for all $n\geq 1$ since $d:E \times E \to [0,\infty)$ is continuous with respect to $\rho \otimes \rho$. On the other hand,
for any $\xi \in C(\nu_1,\nu_2)$  we have
$$
\xi \big( d(X,Y)< \delta_n\big)\le {2\over n},
$$
proving the proposition.
\end{proof}
\section{Jankov's lemma and the proof of Proposition \ref{jan_cor}}

 Recall that a measurable space $(\mathbb{X}, \mathcal{X})$ is called \emph{(standard) Borel} if it is measurably isomorphic to a Polish space equipped with its
 Borel $\sigma$-algebra. For any
Borel space $(\mathbb{X}, \mathcal{X})$ and any set $A\in \mathcal{X}$,  this set endowed with its  {trace $\sigma$-algebra}
   $$
   \mathcal{X}_A:=\{A\cap B, B\in \mathcal{X}\}
   $$
is a Borel measurable space, see \cite[Corollary 13.4]{K95}.

Our proof of Proposition \ref{jan_cor} is based on the following lemma.

\begin{lemma} \emph{(Jankov's lemma,  \cite[Appendix 3 \S 1]{Dynkin_Yushk})}. Let $(\mathbb{X}, \mathcal{X})$, $(\mathbb{Y}, \mathcal{Y})$ be Borel measurable spaces and
let $f:\mathbb{Y}\to \mathbb{X}$ be a measurable mapping with $f(\mathbb{Y})=\mathbb{X}$.

Then for any probability measure $\nu$ on $(\mathbb{X}, \mathcal{X})$ there exists a measurable function $\phi: \mathbb{X}\to \mathbb{Y}$ such that
$f(\phi(x))=x$ for $\nu$-a.a. $x\in \mathbb{X}$.
\end{lemma}

In the framework of Proposition \ref{jan_cor}, we put $\mathbb{X}=M, \mathcal{X}=\EE_M$ (the trace $\sigma$-algebra), then $(\mathbb{X}, \mathcal{X})$ is a Borel space.
We define $\nu$ as the measure $\mu$ conditioned by $M$.

Before proceeding with the construction, we mention  several simple facts we will use.  First, let $\mathbb{S}$ be a Polish space and
 $\mathcal{P}(\mathbb{S})$ be endowed by the corresponding Kantorovich-Rubinshtein metric. Then the subset $\Delta\subset \mathcal{P}(\mathbb{S})$  consisting of all
 $\delta$-measures (that is, measures concentrated in one point) is closed, and  $\mathbb{S}$ and $\Delta$ are isomorphic.
 Second,
 the mapping $\theta$ from  $\mathcal{P}(E^\infty\times E^\infty)$ to $\mathcal{P}(E\times E)$ which maps the law of $\{(X_n, Y_n), n\geq 0\}$ to the law of $(X_0, Y_0)$ is
 (Lipschitz) continuous. Hence, the subset
 $$
 \Xi:=\{\xi\in \mathcal{P}(E^\infty\times E^\infty):\theta(\xi)\hbox{ is a $\delta$-measure}\}
 $$
 is closed. In addition, the mapping $\varrho:\Xi\to E\times E$ which transforms $\xi\in \Xi$ to the (unique) point $(x,y)\in E$ such that $\theta(\xi)=\delta_{(x,y)}$, is
 continuous. Then $\Xi$ endowed with the  trace $\sigma$-algebra is a Borel space and $\varrho$ is a measurable mapping on this space with $\varrho(\Xi)=E\times E$.
 Denote by $\varrho_{1,2}$ the (measurable) mappings  $\Xi\to E$ such that  $\varrho(\xi)=(\varrho_1(\xi), \varrho_2(\xi)), \xi\in \Xi.$

 Now we can proceed with the construction which deduces Proposition \ref{jan_cor} from Jankov's lemma. We fix $x\in E$, put
 $$
 \mathbb{Y}=\{\xi\in \Xi: \varrho_1(\xi)=x, \varrho_2(\xi)\in M, \pi_1(\xi)\sim \P_x, \pi_2(\xi)\ll \P_{\varrho_2(\xi)}\},
 $$
 and  $f(\xi)=\varrho_2(\xi), \xi\in \mathbb{Y}$.   Clearly, $f(\mathbb{Y})=M$ and $f$ is a restriction on $\mathbb{Y}$ of a measurable mapping $\Xi\to E$ (the projection
 of $\varrho$ on the second coordinate). Hence in order to be able to apply Jankov's lemma we need only to show that $\mathbb{Y}$ is a measurable subset of $\Xi$.
 Because $\varrho_{1,2}$ are measurable and $\{x\}, M\in \EE$, the sets
 $$
 \{\xi\in \Xi: \varrho_1(\xi)=x\}, \quad \{\xi\in \Xi: \varrho_2(\xi)\in M\}
$$
are measurable.

Next, recall that for any two probability measures $\P, \Q$ on $(E^\infty, \mathcal{E}^{\otimes \infty})$ one has $\P\ll \Q$ if, an only if, for every $\eps>0$ there
exists $\delta>0$ such that
 $$
 \P(A)\leq \eps\quad \hbox{for any $A\in \mathcal{E}^{\otimes \infty}$ such that}\quad \Q(A)\leq \delta.
 $$
 Because $E^\infty$ is a Polish space, there exists a
 countable algebra
 $\mathcal{A}$ which generates $\mathcal{E}^{\otimes \infty}$, and then for any $\gamma>0, A\in \mathcal{E}^{\otimes \infty}$ there exists $A_\gamma\in \mathcal{A}$
 such that
 $$
 \P(A\triangle A_\gamma)<\gamma, \quad \Q(A\triangle A_\gamma)<\gamma.
 $$
 Then in the above characterization of the absolute continuity the class $\mathcal{E}^{\otimes \infty}$ can be replaced by $\mathcal{A}$. Hence
 $$
\{\xi\in \Xi: \pi_2(\xi)\ll \P_{\varrho_2(\xi)}\}=\bigcap_{m=1}^\infty\bigcup_{k=1}^\infty\bigcap_{A\in \mathcal{A}}B_{m,k}(A),
$$
where
$$
B_{m,k}(A)=\left\{\xi:  \pi_2(\xi)(A)\leq m^{-1}, \P_{\varrho_2(\xi)}(A)\leq k^{-1}\right\}\bigcup \left\{\xi:  \P_{\varrho_2(\xi)}(A)> k^{-1}\right\}.
$$
Since the mappings $\varrho_2:\Xi\to E, \pi_2:\Xi\to \mathcal{P}(E^\infty)$, $E\ni v\mapsto \P_v\in  \mathcal{P}(E^\infty)$, and
$$
\mathcal{P}(E^\infty)\ni \P\mapsto \P(A)\in \mathbb{R}, \quad A\in \mathcal{A}
$$
are measurable, each of the sets $B_{m,k}(A)$ is measurable. Therefore the set
$$
\{\xi\in \Xi: \pi_2(\xi)\ll \P_{\varrho_2(\xi)}\}
$$
is measurable, as well. Finally, a similar and simpler argument shows that the set
$$
\{\xi\in \Xi: \pi_1(\xi)\sim \P_{x}\}
$$
is measurable (we omit the explicit expression for this set here).

Summarizing, we have that $\mathbb{Y}$ is a measurable subset of $\Xi$ and therefore, being endowed with the trace $\sigma$-algebra, is a Borel space.
We finish the proof of Proposition \ref{jan_cor} by applying
Jankov's lemma to the Borel spaces $\mathbb{X}$, $\mathbb{Y}$, the mapping $f$, and  the measure $\nu$ specified above.

\section{Kuratovskii and Ryll-Nardzevski's theorem and the proof of Proposition \ref{prop}} Our proof of Proposition \ref{prop} is based on measurability and
measurable selection results discussed in \cite{Stroock_Varad}, Chapter 12.1. Let us survey the required results briefly.

Let $\mathbb{X}$ be a Polish space with complete metric $\rho$. Denote by $\comp(\mathbb{X})$ the space of all non-empty compact subsets of $\mathbb{X}$, endowed with the
Hausdorff metric.

\begin{theorem}\label{tKR} (\cite[Theorem 12.1.10]{Stroock_Varad}  Let $(E, \EE)$ be a measurable space and $\Phi:E \to \comp(\mathbb{X})$ be a measurable map.
Then there exists a measurable map $\phi: E\to \mathbb{X}$ such that $\phi(q)\in \Phi(q), q\in E$.
\end{theorem}
The above theorem is a weaker version of the Kuratovskii and Ryll-Nardzevski's theorem on measurable selection for a set-valued mapping which takes values in the
space of closed subsets of $\mathbb{X}$; e.g. \cite{W80}.

In the set-up of Proposition \ref{prop}, for $\mu,\nu\in \mathcal{P}(S_2)$,
we denote  by $C_{\mathrm{opt}}(\mu, \nu)$ the subset of $C(\mu, \nu)$ consisting of all  couplings which minimize the distance-like  function $h$; that is,
$$
\eta\in C_{\mathrm{opt}}(\mu, \nu)\quad \Leftrightarrow\quad \eta\in C(\mu, \nu), \quad \int_{S_2\times S_2}h(u,v)\eta(\dd u, \dd v)=h(\mu, \nu).
$$
We prove the following simple facts.
\begin{lemma}
\begin{enumerate} For any $\mu,\nu\in \mathcal{P}(S_2)$:
  \item the set $C_{\mathrm{opt}}(\mu, \nu)$ is non-empty;
                             \item the sets $C(\mu, \nu)$, $C_{\mathrm{opt}}(\mu, \nu)$ are compact.
    \end{enumerate}
    \end{lemma}
\begin{proof}  Since $\pi_1, \pi_2:
\mathcal{P}(S_2\times S_2)\to \mathcal{P}(S_2)$ are continuous, any weak limit point of a sequence from $C(\mu, \nu)$ belongs to $C(\mu, \nu)$. Because the marginal
distributions of all $\eta\in C(\mu, \nu)$ are the same, the set $C(\mu, \nu)$ is tight, which by the Prokhorov theorem completes the proof of compactness of  $C(\mu, \nu)$.

Next, the mapping
$$
\mathcal{P}(S_2\times S_2)\ni \eta\mapsto I_h(\eta):=\int_{S_2\times S_2}h(u,v)\eta(\dd u, \dd v) \in [0,1]
$$
is lower semicontinuous. To see that, consider a sequence $\eta_n\Rightarrow \eta$ and use the Skorokhod ``common probability space principle'': there exist random
elements $X_n, n\geq 1, X$  with $\mathrm{Law}\,(X_n)=\eta_n, \mathrm{Law}\,(X)=\eta$ such that
$X_n\to X$ a.s. (see \cite[Theorem 11.7.2]{Dudley}. Since $h$ is bounded and lower semicontinuous, we  have
$$
\E h(X)\leq \E\liminf_nh(X_n)\leq \liminf_n \E h(X_n),
$$
which proves the required semicontinuity of $I_h$. By this semicontinuity (a) the function $I_h$ attains its minimum on the compact set $C(\mu, \nu)$, i.e.
$C_{\mathrm{opt}}(\mu, \nu)$ is non-empty; (b) the set $C_{\mathrm{opt}}(\mu, \nu)$ is closed, and since it is a subset of the compact set $C(\mu, \nu)$, it is compact.
    \end{proof}

To prove Proposition \ref{prop}, we apply Theorem \ref{tKR} in the following setting:  $E=S_1\times S_1$, $\mathbb{X}=\mathcal{P}(S_2\times S_2)$, and
$$
\Phi\big((x,y)\big)=C_{\mathrm{opt}}(Q(x), Q(y)), \quad (x,y)\in E.
$$
We represent $\Phi$ as a composition of $\Psi$ and $\Upsilon$, where
$$
\Psi\big((x,y)\big)=C(Q(x), Q(y)), \quad (x,y)\in E
$$
and
$$
\Upsilon(K)=\left\{\eta\in \mathbb{X}: I_h(\eta)=\min_{\zeta\in K}I_h(\zeta)\right\}\in \comp(\mathbb{X}), \quad K\in \comp(\mathbb{X}).
$$
Clearly, the minimization of $I_h$ is equivalent to maximization of $1-I_h$, and $1-I_h$ is upper semicontinuouus. Hence the mapping
$\Upsilon:\comp(\mathbb{X})\to \comp(\mathbb{X})$ is measurable by \cite{Stroock_Varad}, Lemma 12.1.7. On the other hand, for any sequence
$(x_n,y_n)\to (x,y)$ and $\eta_n\in \Psi((x_n, y_n))$ we have that the marginal distributions of $\eta_n$ weakly converge to
$Q(x), Q(y)$ respectively. Then by the Prokhorov theorem there exist a weakly convergent subsequence $\eta_{n_k}$, and in addition the weak
limit has the marginal distributions $Q(x), Q(y)$, that is, belongs to $\Psi((x,y))$. Then  the mapping $\Psi:E\to \comp(\mathbb{X})$ is
measurable by \cite[Lemma 12.1.8]{Stroock_Varad}. Hence $\Phi$ is measurable, as well, and we obtain the statement of Proposition \ref{prop} as a straightforward
corollary of Theorem \ref{tKR}. \qed

\end{document}